\documentclass[a4paper,12pt,reqno]{amsart}

\usepackage[utf8]{inputenc}  \usepackage[T1]{fontenc}   
\usepackage[english]{babel}
\usepackage{amsmath}
\usepackage{amsthm}
\usepackage{amsfonts} 
\usepackage{amssymb}
\usepackage{mathtext}
\usepackage{graphicx}
\usepackage{hyperref} 
\usepackage{array}
\usepackage{adjustbox}
\usepackage{booktabs}
\usepackage{multirow}
\usepackage{caption}
\usepackage{subcaption}

\newcommand{\N} {\mathbb{N}}

\newtheorem{theo}{Theorem}[section]
\newtheorem{conj}{Conjecture}[section]

\newtheorem{prop}{Proposition}[section]
\newtheorem{lem}{Lemma}[section]

{\theoremstyle{definition} \newtheorem{defi}{Definition}[section]}
{\theoremstyle{definition} }
{\theoremstyle{definition} }
{\theoremstyle{remark} \newtheorem{rem}{Remark}[section]}
{\theoremstyle{definition} }

\renewenvironment{proof}{\noindent{\bf Proof.}}{\qed}

\author{Ange Bigeni}
\address{Institut Camille Jordan, Universit\'e Claude Benard Lyon 1 (France)
}
\email{bigeni@math.univ-lyon1.fr}
\title[A bijection between irreducible $k$-shapes and surjective pistols]{A bijection between irreducible $k$-shapes the surjective pistols of height $k-1$}
\keywords{Genocchi numbers; Gandhi polynomials; (irreducible) $k$-shapes; surjective pistols}
\begin{document}
\maketitle
\begin{abstract}
This paper constructs a bijection between irreducible $k$-shapes and surjective
pistols of height $k-1$, which carries the "free $k$-sites" to the
fixed points of surjective pistols. The bijection confirms a conjecture of
Hivert and Mallet (FPSAC 2011) that the number of irreducible $k$-shape is counted by the
Genocchi number $G_{2k}$.
\end{abstract}

\section{Introduction}
\label{sec:intro}
\hspace*{-5.9mm} The study of $k$-shapes arises naturally in the combinatorics of $k$-Schur
functions (see \cite{LLMS}). In a 2011 FPSAC paper, Hivert and Mallet showed that the generating function of all $k$-shapes was a rational function whose numerator $P_k(t)$ was defined in terms of what they called irreducible $k$-shapes. The sequence of numbers of irreducible $k$-shapes $(P_k(1))_{k \geq 1}$ seemed to be the sequence of Genocchi numbers $(G_{2k})_{k \geq 1} = (1,1,3,17,155,2073,\hdots)$~\cite{genocchi}, which may be defined by $G_{2k} = Q_{2k-2}(1)$ for all $k \geq 2$ (see \cite{Carlitz,Riordan}) where $Q_{2n}(x)$ is the Gandhi polynomial defined by the recursion $Q_2(x) =x^2$ and 
\begin{equation} \label{inductionformulagandhi} 
Q_{2k+2}(x) = x^2(Q_{2k}(x+1)-Q_{2k}(x)). \end{equation}
Hivert and
Mallet defined a statistic $fr(\lambda)$ counting the so-called free $k$-sites on 
the partitions $\lambda$ in the set of irreducible $k$-shapes $IS_k$, and
conjectured that 
\begin{equation} \label{equationgandhi}
Q_{2k-2}(x) = \sum_{\lambda \in IS_k} x^{fr(\lambda)+2}.
\end{equation}
The goal of this paper is to construct a bijection between irreducible $k$-shapes and surjective pistols of height $k-1$, such that every free $k$-site of an irreducible $k$-shape is carried to a fixed point of the corresponding surjective pistol. Since the surjective pistols are known to generate the Gandhi polynomials with respect to the fixed points (see Theorem \ref{theoremegandhi}), this bijection will imply Formula~(\ref{equationgandhi}).\\
The rest of this paper is organized as follows. In Section \ref{sec:background}, we give some background about surjective pistols (in Subsection \ref{sec:surjectivepistols}), partitions, skew partitions and $k$-shapes (in Subsection \ref{sec:backgroundpartitions}), then we focus on irreducible $k$-shapes (in Subsection \ref{sec:irreduciblekshapes}) and enounce Conjecture \ref{conjecturepistol} raised by Mallet (which implies Formula \ref{equationgandhi}), and the main result of this paper, Theorem \ref{theoremebijection}, whose latter conjecture is a straight corollary. In Section \ref{sec:preliminaires}, we give preliminaries of the proof of Theorem \ref{theoremebijection} by introducing the notion of partial $k$-shapes. In Section \ref{sec:proof}, we demonstrate Theorem \ref{theoremebijection} by defining two inverse maps $\varphi$ (in Subsection \ref{sec:varphi}) and $\phi$ (in Subsection \ref{sec:phi}) which connect irreducible $k$-shapes and surjective pistols and keep track of the two statistics. Finally, in Section \ref{sec:workinprogress}, we explore the corresponding interpretations of some generalizations of the Gandhi polynomials, generated by the surjective pistols with respect to refined statistics, on the irreducible $k$-shapes.

\section{Definitions and main result}
\label{sec:background}

\subsection{Surjective pistols}
\label{sec:surjectivepistols}

For all positive integer $n$, we denote by $[n]$ the set $\{1,2,\hdots,n\}$. A \textit{surjective pistol of height $k$} is a surjective map $f : [2k] \rightarrow \{2,4,\hdots,2k\}$ such that $f(j) \geq j$ for all $j \in [2k]$. We denote by $SP_{k}$ the set of surjective pistols of height $k$. By abuse of notation, we assimilate a surjective pistol $f \in SP_{k}$ into the sequence $(f(1),f(2),\hdots, f(2k))$.
A \textit{fixed point} of $f \in SP_k$ is an integer $j \in [2k]$ such that $f(j) = j$. We denote by $fix(f)$ the number of fixed points different from $2k$ (which is always a fixed point). A surjective pistol of height $k$ can also be seen as a tableau made of $k$ right-justified rows of length $2,4,6,\hdots,2k$ (from top to bottom), such that each row contains at least one dot, and each column contains exactly one dot. The map $f$ corresponding to such a tableau would be defined as $f(j) = 2(\lceil j/2 \rceil + z_j)$ where the $j$-th column of the tableau contains a dot in its $(1+z_j)$-th cell (from top to bottom) for all $j \in [2k]$. For example, if $f = (2,4,4,8,8,6,8,8) \in SP_4$, the tableau corresponding to $f$ is depicted in Figure \ref{exempletableau}.
\begin{figure}[!h]
\centering
\includegraphics[width=4cm]{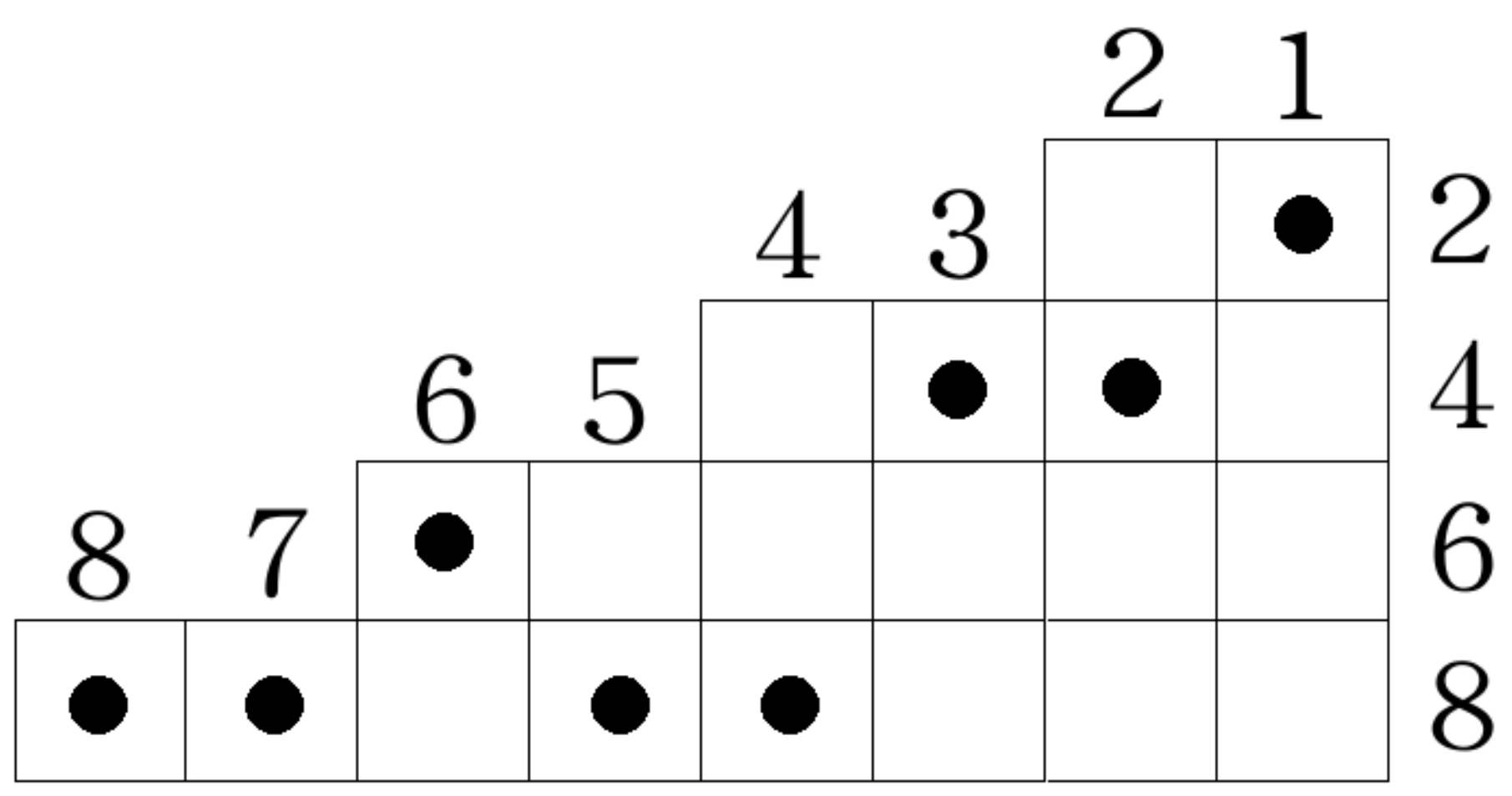}
\caption{Tableau corresponding to $f = (2,4,4,8,8,6,8,8) \in SP_4$.}
\label{exempletableau}
\end{figure}
\\
In particular, an integer $j = 2i$ is a fixed point of $f$ if and only if the dot of the $2i$-th column of the corresponding tableau is at the top of the column. For example, the surjective pistol $f$ of Figure \ref{exempletableau} has $2$ fixed points $6$ and $8$, but $fix(f) = 1$ (because the fixed point $2k=8$ is not counted by the statistic). The following result is due to Dumont.

\begin{theo}[\cite{Dumont2}] \label{theoremegandhi}
For all $k \geq 2$, the Gandhi polynomial $Q_{2k}(x)$ has the following combinatorial interpretation:
$$Q_{2k}(x) = \sum_{f \in SP_{k}} x^{fix(f)+2}.$$
\end{theo} 



\subsection{Partitions, skew partitions, $k$-shapes}
\label{sec:backgroundpartitions}

A partition is a a finite sequence of positive integers $\lambda = (\lambda_1,\lambda_2,\hdots,\lambda_m)$ such that $\lambda_1 \geq \lambda_2 \geq \hdots \geq \lambda_m$. By abuse of definition, we consider that a partition may be empty (corresponding to $m=0$). A convenient way to visualize a partition $\lambda = (\lambda_1,\lambda_2,\hdots,\lambda_m)$ is to consider its Ferrers diagram (denoted by $[\lambda]$), which is composed of cells organized in left-justified rows such that the $i$-th row (from bottom to top) contains $\lambda_i$ cells. The \emph{hook length} of a cell $c$ is defined as the number of cells located to its right in the same row (including $c$ itself) or above it in the same column. If the hook length of a cell $c$ equals $h$, we say that $c$ is hook lengthed by the integer $h$. For example, the Ferrers diagram of the partition $\lambda = (4,2,2,1)$ is represented in Figure \ref{exempleferrers}, in which every cell is labeled by its own hook length.
\begin{figure}[!h]
\centering
\begin{minipage}{.5\textwidth}
  \centering
\includegraphics[width=2cm]{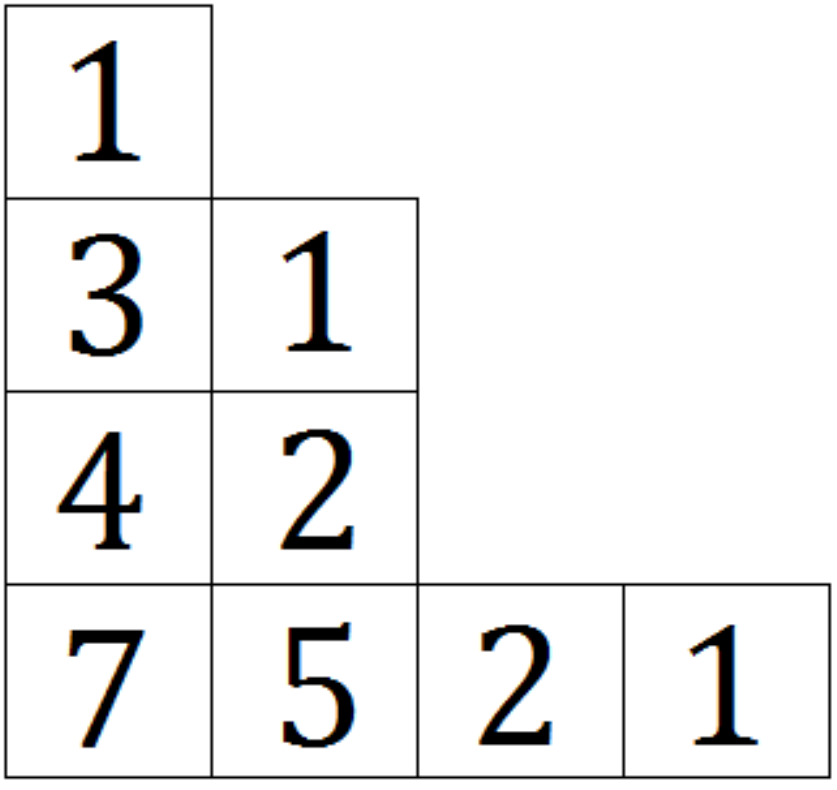}
\caption{Ferrers diagram of the partition $\lambda=(4,2,2,1)$.}
\label{exempleferrers}
\end{minipage}%
\begin{minipage}{.5\textwidth}
  \centering
\includegraphics[width=2cm]{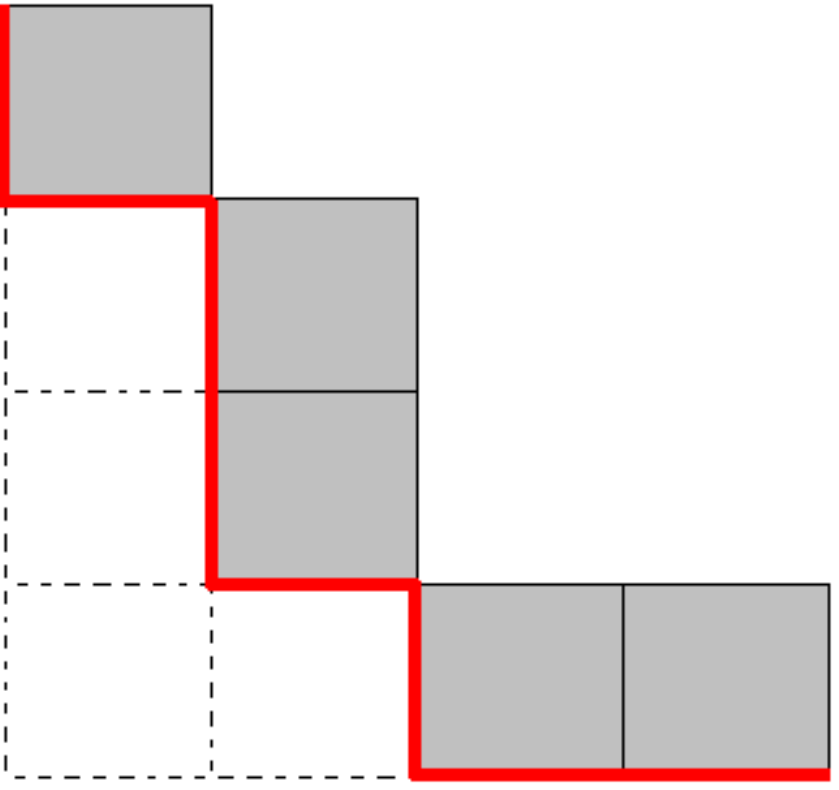}
\caption{Skew partition $\lambda \backslash \mu$.}
\label{exempleskew}
\end{minipage}
\end{figure}
\\
We will sometimes assimilate partitions with their Ferrers diagrams.
If two partitions $\lambda = (\lambda_1,\hdots,\lambda_p)$ and $\mu = (\mu_1,\hdots,\mu_q)$ (with $q \leq p)$ are such that $\mu_i \leq \lambda_i$ for all $i \leq q$, then we write $\mu \subseteq \lambda$ and we define the \textit{skew partition} $s = \lambda/\mu$ as the diagram $[\lambda] \backslash [\mu]$, the Ferrers diagram $[\mu]$ appearing naturally in $[\lambda]$. For example, if $\lambda = (4,2,2,1)$ and $\mu = (2,1,1)$, then $\mu \subseteq \lambda$ and the skew partition $\lambda \backslash \mu$ is the diagram depicted in Figure \ref{exempleskew}.
For all skew partition $s$, we name \textit{row shape} (respectively \textit{column shape}) of $s$, and we denote by $rs(s)$ (resp. $cs(s)$), the sequence of the lengths of the rows from bottom to top (resp. the sequence of the heights of the columns from left to right) of $s$. Those sequences are not necessarily partitions. For example, if $s$ is the skew partition depicted in Figure \ref{exempleskew}, then $rs(s) = (2,1,1,1)$ and $cs(s) = (1,2,1,1)$ (in particular $cs(s)$ is not a partition). If the lower border of $s$ is continuous, \textit{i.e.}, if it is not fragmented into several pieces,, we also define a canonical partition $<s>$ obtained by inserting cells in the empty space beneath every column and on the left of every row of $s$. For example, if $s$ is the skew partition depicted in Figure \ref{exempleskew}, the lower border of $s$ is drawed as a thin red line which is continuous, and $<s>$ is simply the original partition $\lambda = (4,2,2,1)$.\\
Now, consider a positive integer $k$. For all partition $\lambda$, it is easy to see that the diagram composed of the cells of $[\lambda]$ whose hook length does not exceed $k$, is a skew partition, that we name \textit{$k$-boundary} of $\lambda$ and denote by $\partial^k(\lambda)$. Incidentally, we name \textit{$k$-rim} of $\lambda$ the lower border of $\partial^k(\lambda)$ (which may be fragmented), and we denote by $rs^k(\lambda)$ (respectively $cs^k(\lambda)$) the sequence $rs(\partial^k(\lambda))$ (resp. the sequence $cs(\partial^k(\lambda))$. For example, the $2$-boundary of the partition $\lambda = (4,2,2,1)$ depicted in Figure \ref{exempleferrers}, is in fact the skew partition of Figure \ref{exempleskew}. Note that if the $k$-rim of $\lambda$ is continuous, then the partition $<\partial^k(\lambda)>$ is simply $\lambda$.
\begin{defi}[\cite{LLMS}] A $k$-shape is a partition $\lambda$ such that the sequences $rs^k(\lambda)$ and $cs^k(\lambda)$ are also partitions.
\end{defi} 
\hspace*{-5.9mm} For example, the partition $\lambda = (4,2,2,1)$ depicted in Figure \ref{exempleferrers} is not a $2$-shape since $cs^2(\lambda) = (1,2,1,1)$ is not a partition, but it is a $k$-shape for any $k \geq 4$ (for instance $cs^5(\lambda) = (3,3,1,1)$ and $rs^5(\lambda) = (3,2,2,1)$ are partitions, so $\lambda$ is a $5$-shape, see Figure \ref{exemplegrid}). Note that the $k$-rim of a $k$-shape $\lambda$ is necessarily continuous, thence $\lambda = < \partial^k(\lambda) >$. Consequently, we will sometimes assimilate a $k$-shape into its $k$-boundary.

\subsection{Irreducible $k$-shapes}
\label{sec:irreduciblekshapes}

Let $\lambda$ be a $k$-shape and $(u,v)$ a pair of positive integers. Following \cite{HM}, we denote by $H_u(\lambda)$ (respectively $V_v(\lambda)$) the set of all cells of the skew partition $\partial^k(\lambda)$ that are contained in a row of length $u$ (resp. the set of all cells of $\partial^k(\lambda)$ that are contained in a column of height $v$).
For example, consider the $5$-shape $\lambda = (4,2,2,1)$.
\begin{figure}[!h]
\centering
\begin{minipage}{.5\textwidth}
  \centering
\includegraphics[width=3cm]{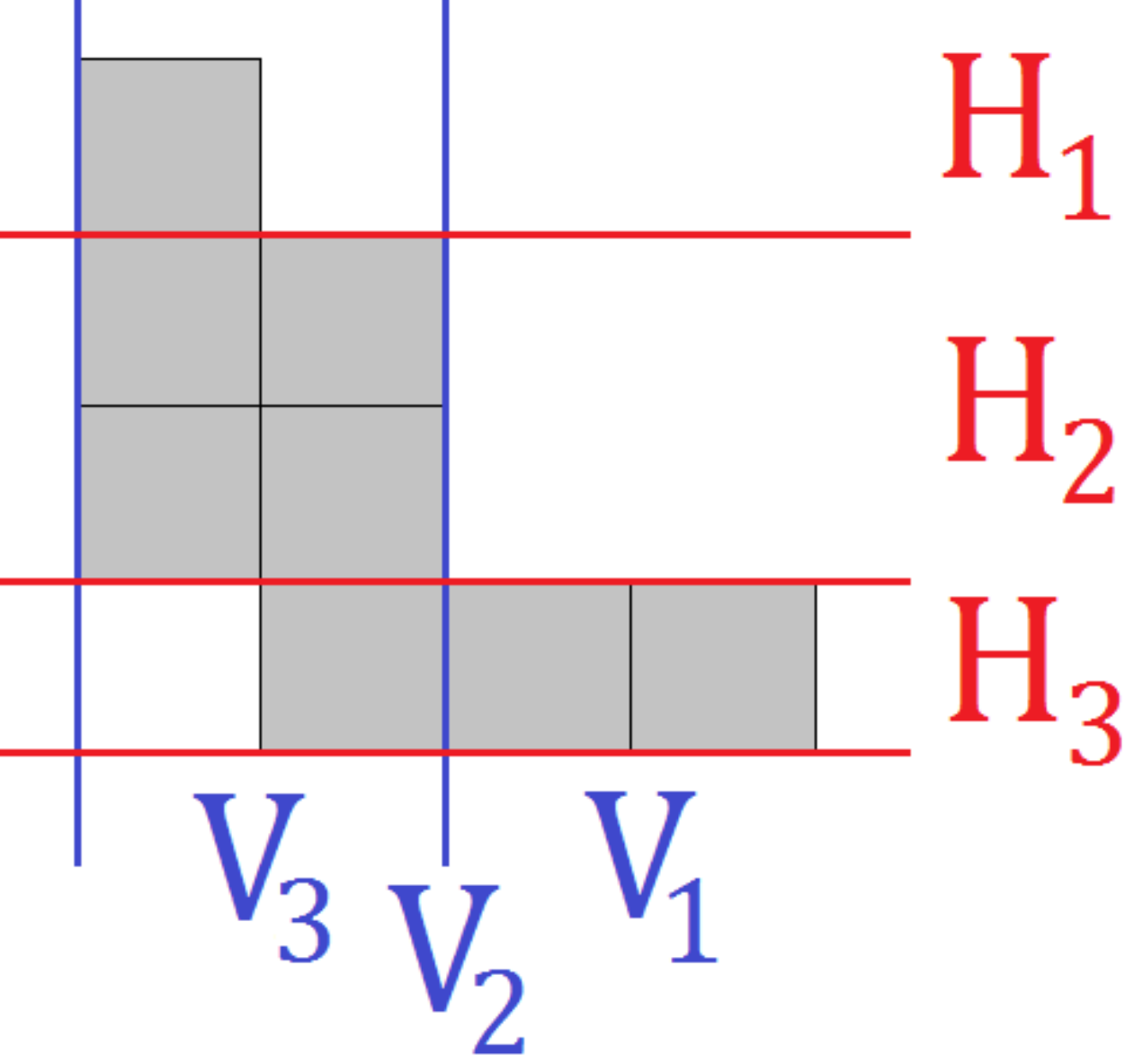}
\caption{Skew partition $\partial^5(\lambda)$ with $\lambda = (4,2,2,1)$.}
\label{exemplegrid}
\end{minipage}
\end{figure}
The sets $(H_u(\lambda))_{u \geq 1}$ and $(V_v(\lambda))_{v \geq 1}$ are outlined in Figure \ref{exemplegrid} (in this example the set $V_2(\lambda)$ is empty). Note that for all $k$-shape $\lambda$ and for all pair of positive integers $(u,v)$, if the set $H_u(\lambda) \cap V_v(\lambda)$ is not empty, then there exists a cell in $V_v(\lambda)$ hook lengthed by at least $u+v-1$. Consequently, if $u+v > k+1$, then by definition of $\partial^k(\lambda)$ the set $H_u(\lambda) \cap V_v(\lambda)$ must be empty.
\\
Hivert and Mallet~\cite{HM} defined an operation which consists in inserting, in a $k$-shape, a $l$-rectangle (namely, a partition whose Ferrers diagram is a rectangle and whose largest hook length is $l$) with $l \in \{k-1,k\}$, the result of the operation being a new $k$-shape. They defined \textit{irreducible $k$-shapes} as $k$-shapes that cannot be obtained in such a way. In this paper, we use an equivalent definition in view of the Proposition 3.8 of \cite{HM}.

\begin{defi}[\cite{HM}]
An \textit{irreducible $k$-shape} is a $k$-shape $\lambda$ such that the sets $H_{i}(\lambda) \cap V_{k-i}(\lambda)$ and $H_{i}(\lambda) \cap V_{k+1-i}(\lambda)$ contain at most $i-1$ horizontal steps of the $k$-rim of $\lambda$ for all $i \in [k]$. We denote by $IS_k$ the set of irreducible $k$-shapes.
\end{defi}

\hspace*{-5.9mm} For example, the $5$-shape $\lambda = (4,2,2,1)$ (see Figure \ref{exemplegrid}) is irreducible: the sets $H_i(\lambda) \cap V_{5-i}(\lambda)$ and $H_j(\lambda) \cap V_{6-j}(\lambda)$ are empty if $i \neq 2$ and $j \neq 3$, and the two sets $H_2(\lambda) \cap V_{3}(\lambda)$ and $H_3(\lambda) \cap V_{3}(\lambda)$ contain respectively $1<2$ and $1<3$ horizontal steps of the $k$-rim of $\lambda$. \\
In general, it is easy to see that for any $k$-shape $\lambda$ to be irreducible, the sets $H_1(\lambda) \cap V_k(\lambda)$ and $H_k(\lambda) \cap~V_1(\lambda)$ must be empty, and by definition the set $H_1(\lambda) \cap V_{k-1}(\lambda)$ must contain no horizontal step of the $k$-rim of $\lambda$. In particular, for $k=1$ or $2$ there is only one irreducible $k$-shape: the empty partition.

\begin{defi}[\cite{HM,HM2}]
Let $\lambda$ be an irreducible $k$-shape with $k \geq 3$. For all $i \in [k-2]$, we say that the integer $i$ is a \textit{free $k$-site} of $\lambda$ if the set $H_{k-i}(\lambda) \cap V_{i+1}(\lambda)$ is empty. We define $\overrightarrow{fr}(\lambda)$ as the vector $(t_1,t_2,\hdots,t_{k-2}) \in \{0,1\}^{k-2}$ where $t_i = 1$ if and only if $i$ is a free $k$-site of $\lambda$. We also define $fr(\lambda)$ as $\sum_{i=1}^{k-2} t_i$ (the quantity of free $k$-sites of $\lambda$).
\end{defi}

\hspace*{-5.9mm} For example, the irreducible $5$-shape $\lambda = (4,2,2,1)$ depicted in Figure \ref{exemplegrid} is such that $\overrightarrow{fr}(\lambda) = (1,0,1)$.
In order to prove the conjecture of Formula \ref{equationgandhi}, and in view of Theorem \ref{theoremegandhi},
Hivert and Mallet proposed to construct a bijection $\phi : IS_k \rightarrow SP_{k-1}$ such that $fix(\phi(\lambda)) = fr(\lambda)$ for all $\lambda$. Mallet~\cite{HM2} refined the conjecture by introducing a vectorial version of the statistic of fixed points: for all $f \in SP_{k-1}$, we define $\overrightarrow{fix}(f)$ as the vector $(t_1,\hdots,t_{k-2}) \in \{0,1\}^{k-2}$ where $t_i = 1$ if and only if $f(2i) = 2i$ (in particular $fix(f) = \sum_i t_i$).

\begin{conj}[\cite{HM2}] \label{conjecturepistol}
For all $k \geq 3$ and $\overrightarrow{v} = (v_1,v_2,\hdots,v_{k-2}) \in \{0,1\}^{k-2}$, the number of irreducible $k$-shapes $\lambda$ such that $\overrightarrow{fr}(\lambda) = \overrightarrow{v}$ is the number of surjective pistols $f \in SP_{k-1}$ such that $\overrightarrow{fix}(f) = \overrightarrow{v}$.
\end{conj}

\hspace*{-5.9mm} The main result of this paper is the following theorem, which implies immediately Conjecture \ref{conjecturepistol}.

\begin{theo} \label{theoremebijection}
There exists a bijection $\varphi : SP_{k-1} \rightarrow IS_k$ such that $\overrightarrow{fr}(\varphi(f)) = \overrightarrow{fix}(f)$ for all $f \in SP_{k-1}$.
\end{theo}

\hspace*{-5.9mm} We intend to demonstrate Theorem \ref{theoremebijection} in the following two sections \S \ref{sec:preliminaires} and \S \ref{sec:proof}.

\section{Partial $k$-shapes}
\label{sec:preliminaires}

\begin{defi}[Labeled skew partitions, partial $k$-shapes and saturation property]
A \textit{labeled skew partition} is a skew partition $s$ whose columns are labeled by the integer $1$ or $2$. If $cs(s)$ is a partition and if the hook length of every cell of $s$ doesn't exceed $k$ (resp. $k-1$) when the cell is located in a column labeled by $1$ (resp. by $2$), we say that $s$ is a \textit{partial $k$-shape}. In that case, if $C_0$ is a column labeled by $1$ which is rooted in a row $R_0$ (\textit{i.e.}, whose bottom cell is located in $R_0$) whose top left cell is hook lengthed by $k$, we say that $C_0$ is saturated. For all $i \in [k-1]$, if every column of height $i+1$ and label $1$ is saturated in $s$, we say that $s$ is saturated in $i$. If $s$ is saturated in $i$ for all $i$, we say that $s$ is saturated.
\end{defi}

\hspace*{-5.9mm} We represent labeled skew partitions by painting in dark blue columns labeled by $1$, and in light blue columns labeled by $2$. For example, the skew partition depicted in Figure \ref{exemplelabel} is a partial $6$-shape, which is not saturated because its unique column $C$ labeled by $1$ is rooted in a row whose top left cell (which is in this exemple the own bottom cell of $C$) is hook lengthed by $5$ instead of $6$.
\begin{figure}[!h]
\centering
\begin{minipage}{.35\textwidth}
  \centering
\includegraphics[width=1.5cm]{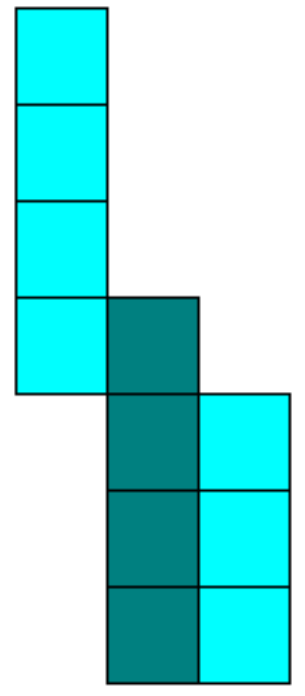}
\caption{Partial $6$-shape $s$.}
\label{exemplelabel}
\end{minipage}%
\begin{minipage}{.35\textwidth}
  \centering
\includegraphics[width=2.3cm]{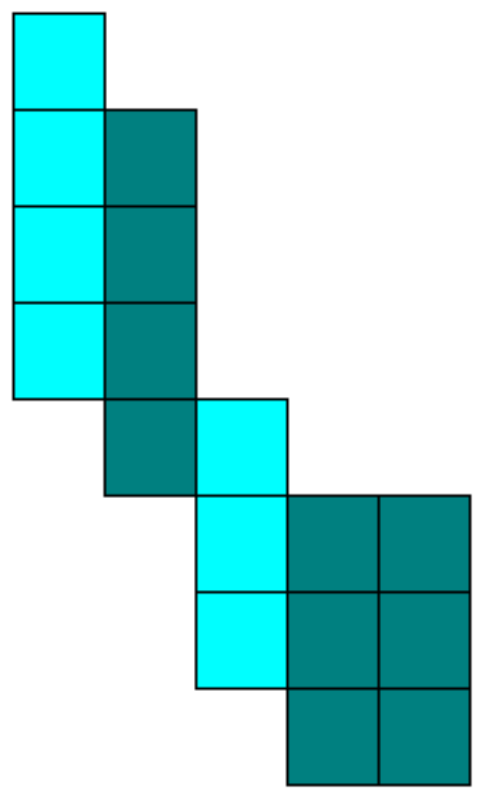}
\caption{Partial $6$-shape $s \oplus^6_1 3^2$.}
\label{exemplelabel2}
\end{minipage}%
\begin{minipage}{.35\textwidth}
  \centering
\includegraphics[width=5cm]{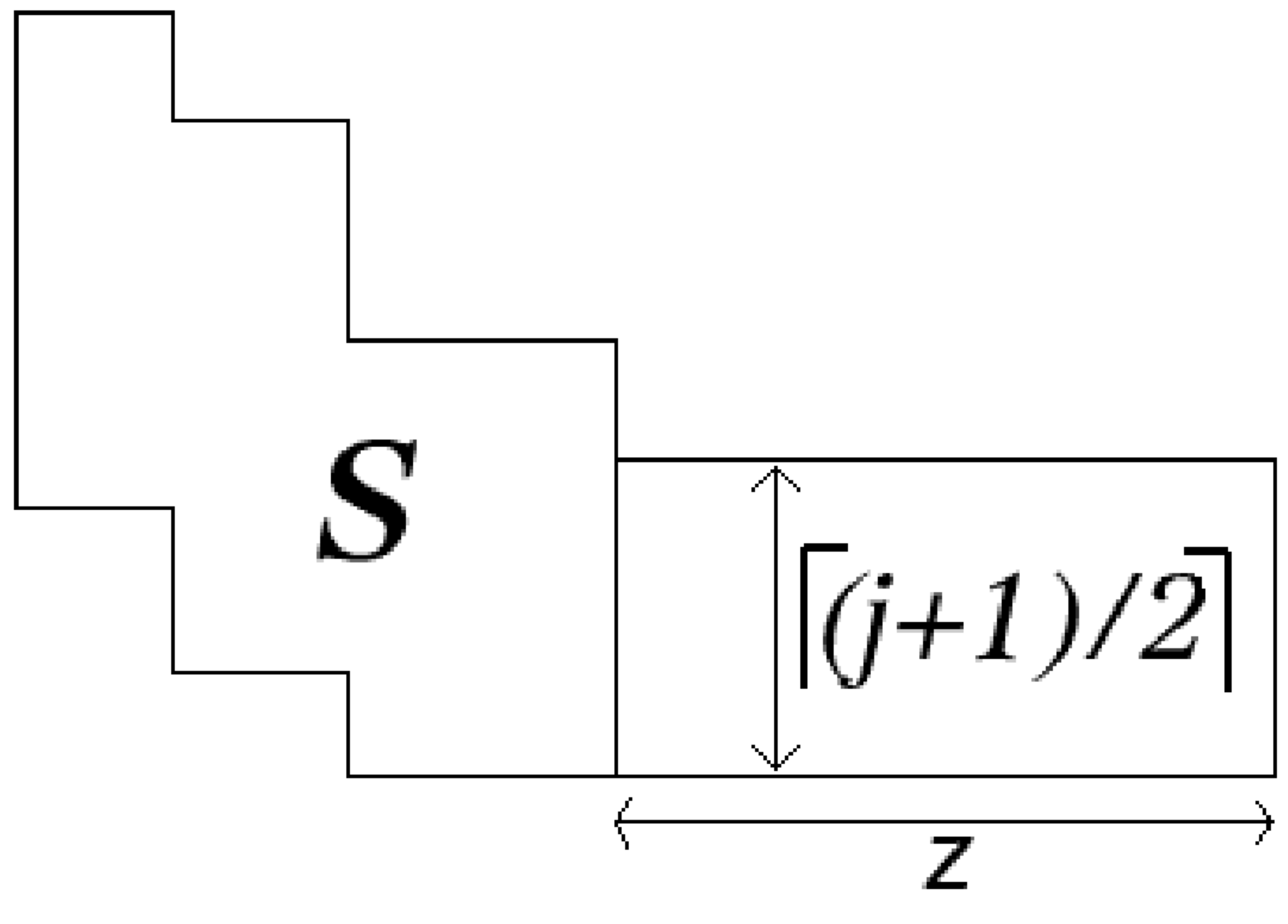}
\caption{Gluing of the rectangle $\lceil (j+1)/2 \rceil^{z_j}$ to $s$.}
\label{exemplegluing}
\end{minipage}%
\end{figure}

\begin{defi}[Sum of partial $k$-shapes with rectangles]
\label{sumskew}
Let $s$ be a partial $k$-shape, and $j \geq 1$ such that the height of every column of $s$ is at least $\lceil (j+2)/2 \rceil$ (if $s$ is the empty skew partition we impose no condition on $j$). Let $z$ be a nonnegative integer and $t(j)$ the integer defined as $1$ if $j$ is even and $2$ if $j$ is odd. We consider the labeled skew partition $\tilde{s}$ obtained by gluing right on the last column of $s$, the amount of $z$ columns of height $\lceil (j+1)/2 \rceil$ (see Figure \ref{exemplegluing}) labeled by the integer $t(j)$.
We apply the following algorithm on $\tilde{s}$ as long as one of the three corresponding conditions is satisfied.
\begin{enumerate}
\item If there exists a column $C_0$ labeled by $1$ (respectively by $2$) in $\tilde{s}$ such that the bottom cell $c_0$ of $C_0$ is a corner of $\tilde{s}$ (a cell of $ \tilde{s}$ with no other cell beneath it or on the left of it)  whose hook length $h$ exceeds $k$ (resp. $k-1$), then we \textit{lift} the column $C_0$, \textit{i.e.}, we erase $c_0$ and we draw a cell on the top of $C_0$ (see Figure \ref{exemplecrossing}).
\item If there exists a column $C_0$ of height $i_0+1$ (with $i_0 \in [k-2]$) and labeled by $1$ in $\tilde{s}$, such that the bottom cell $c_0$ of $C_0$ is on on the right of the bottom cell of a column whose height is not $i_0+1$ or whose label is not $1$, then we lift every column on the left of $C_0$ whose bottom cell is located in the same row as $c_0$, \textit{i.e.}, we erase every cell on the left of $c_0$ and we draw a cell on every corresponding column (see Figure \ref{exempleliftcorner}).
\item If there exists a column $C_0$ of height $i_0+1$ (with $i_0 \in [k-2]$) and labeled by $1$ in $\tilde{s}$, such that the bottom cell $c_0$ of $C_0$ is a corner whose hook length $h$ doesn't equal $k$ (which means it is rooted in a row $R_0$ of $\tilde{s}$ whose length is $k-i_0 - l < k-i_0$ for some $l \geq 1$), whereas this hook length was exactly $k$ in the partial $k$-shape $s$, then we lift every column rooted in the same row as the $l$ last columns (from left to right) intersecting $R_0$, in such a way the hook length of $c_0$ becomes $k$ again in $\tilde{s}$ (see Figure \ref{exempleliftconservation}).
\end{enumerate}
\begin{figure}[!h]
\centering
\begin{minipage}{.3\textwidth}
  \centering
\includegraphics[width=3cm]{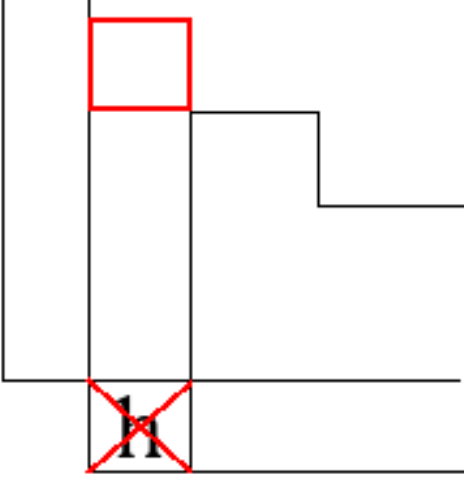}
\caption{}
\label{exemplecrossing}
\end{minipage}%
\begin{minipage}{.3\textwidth}
  \centering
\includegraphics[width=3.5cm]{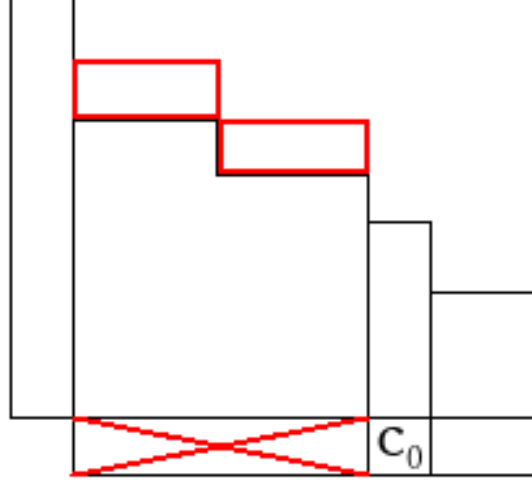}
\caption{}
\label{exempleliftcorner}
\end{minipage}%
\begin{minipage}{.3\textwidth}
  \centering
\includegraphics[width=3.5cm]{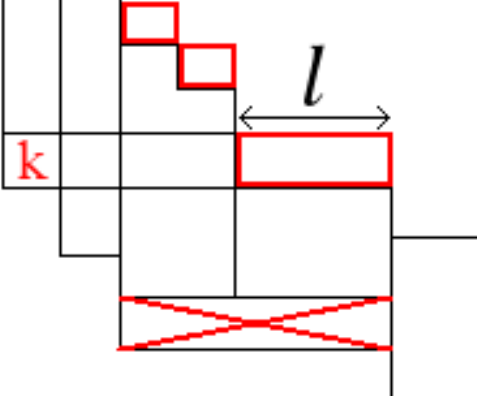}
\caption{}
\label{exempleliftconservation}
\end{minipage}%
\end{figure}
It is easy to see that this algorithm is finite and that the final version of $\tilde{s}$ is a partial $k$-shape, that we define as the $t(j)$-sum of the partial $k$-shape $s$ with the rectangle $\lceil (j+1)/2 \rceil^{z}$ (the partition whose Ferrers diagram is a rectangle of length $z$ and height $\lceil (j+1)/2 \rceil$), and which we denote by
$$s \oplus^k_{t(j)} \lceil (j+1)/2 \rceil^{z}.$$
\end{defi}

\hspace*{-5.9mm} For example, the $1$-sum $s \oplus_1^6 3^2$ of the partial $6$-shape $s$ represented in Figure \ref{exemplelabel}, with the rectangle composed of $2$ columns of height $3$ and label $1$, is the partial $6$-shape depicted in Figure \ref{exemplelabel2}.

\begin{rem} \label{remarquesaturation}
In the context of Definition \ref{sumskew}, the rule $(3)$ of the latter definition guarantees that any saturated column of $s$ is still saturated in $s \oplus_{t(j)}^k \lceil (j+1)/2 \rceil^z$. In particular, if $s$ is saturated in $i \in [k-2]$, then $s \oplus_{t(j)}^k \lceil (j+1)/2 \rceil^z$ is also saturated in $i$.
 \end{rem}
 
 \begin{lem} \label{detaillifting1} 
Let $s$ be a partial $k$-shape, let $j \in [2k-4]$ such that every column of $s$ is at least $\lceil (j+2)/2 \rceil$ cells high, and let $z \in \{0,1,\hdots,k-1-\lceil j/2 \rceil \}$. We consider two consecutive columns (from left to right) of $s$, which we denote by $C_1$ and $C_2$, with the same height and the same label but not the same level, and such that $C_1$ has been lifted in the context $(1)$ of Definition \ref{sumskew} (note that it cannot be in the context $(2)$). If $C_2$ has been lifted at the same level as $C_1$ in $s \oplus_{t(j)}^k \lceil (j+1)/2 \rceil^z$, then it is not in the context $(1)$ of Definition \ref{sumskew}.
\end{lem}

\begin{proof}
Let $R_1$ (resp. $R_2$) be the row in which $C_1$ (resp. $C_2$) is rooted, and let $R$ be the row beneath $R_1$. Let $l$ be the length of $R$. Since $C_1$ and $C_2$ have the same height and the same label, and since $C_1$ has been lifted in the context $(1)$ of Definition \ref{sumskew}, then it is necessary that the length of $R_2$ equals $l$ as well. Consequently, the partial $k$-shape $s$ is like depicted in Figure~\ref{figuredetail1}. We also consider the last column $C_3$ to be rooted in $R_2$, and the column $C_4$ on the right of $C_3$.
\begin{figure}[!h]
\centering
\includegraphics[width=5cm]{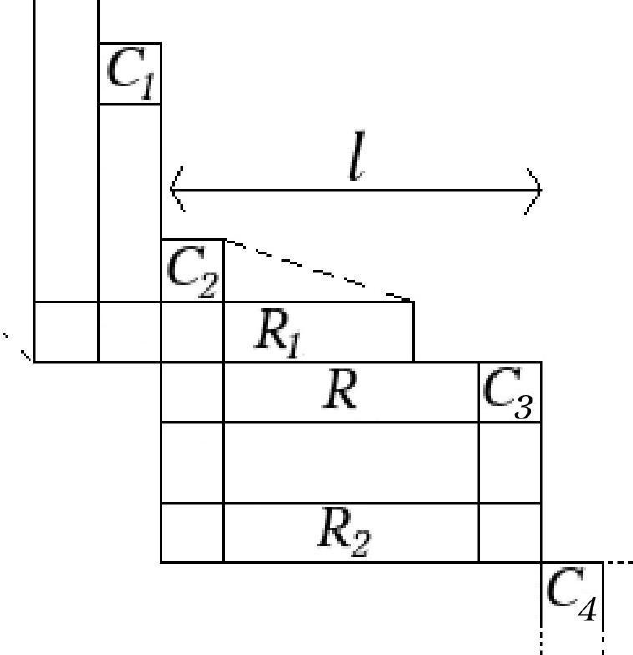}
\caption{Partial $k$-shape $s$.}
\label{figuredetail1}
\end{figure}
\\ 
Now, suppose that, in $s \oplus_{t(j)}^k \lceil (j+1)/2 \rceil^z$, the column $C_2$ has been lifed at the same level as $C_1$ in the context $(1)$ of Definition \ref{sumskew}. To do so, it is necessary that the row $R$ gains cells between $s$ and $s \oplus_{t(j)}^k \lceil (j+1)/2 \rceil^z$, \textit{i.e.}, that the column $C_4$ is lifted at the same level as $C_3$. By hypothesis, it means that $C_4$ must be lifted down to at least $\lceil (j+2)/2 \rceil$ cells between $s$ and $s \oplus_{t(j)}^k \lceil (j+1)/2 \rceil^z$. Obviously, every column of $s \oplus_{t(j)}^k \lceil (j+1)/2 \rceil^z$ has been lifted up to at most $\lceil (j+1)/2 \rceil$ cells, so $\lceil (j+1)/2 \rceil = \lceil (j+2)/2 \rceil$, \textit{i.e.}, there exists $p \in [k-2]$ such that $j = 2p$. Consequently, the partial $k$-shape $s \oplus_{t(j)}^k \lceil (j+1)/2 \rceil^z$ is obtained by adding $z$ columns of height $p+1$ and label $1$ to $s$. However, according to the rule $(3)$ of Definition \ref{sumskew}, only the $p$ top cells of those $z$ columns may lift the columns of $s$ in the context $(1)$, \textit{i.e.}, the columns of $s$ are lifted up to at most $p$ cells in this context, thence $C_4$ cannot be lifted at the same level as $C_3$, which is absurd.
\end{proof}
 
\begin{rem} \label{explicitation3}
Here, we give precisions about context $(3)$ of Definition \ref{sumskew}. Using the same notations, consider the column $C_1$ of $s$ which contains the last cell (from left to right) of $R_0$, and $C_2$ the column which follows $C_1$ (see Figure \ref{figureexplicitation31}).
\begin{figure}[!h]

\centering
\begin{minipage}{.5\textwidth}
  \centering
\includegraphics[width=6cm]{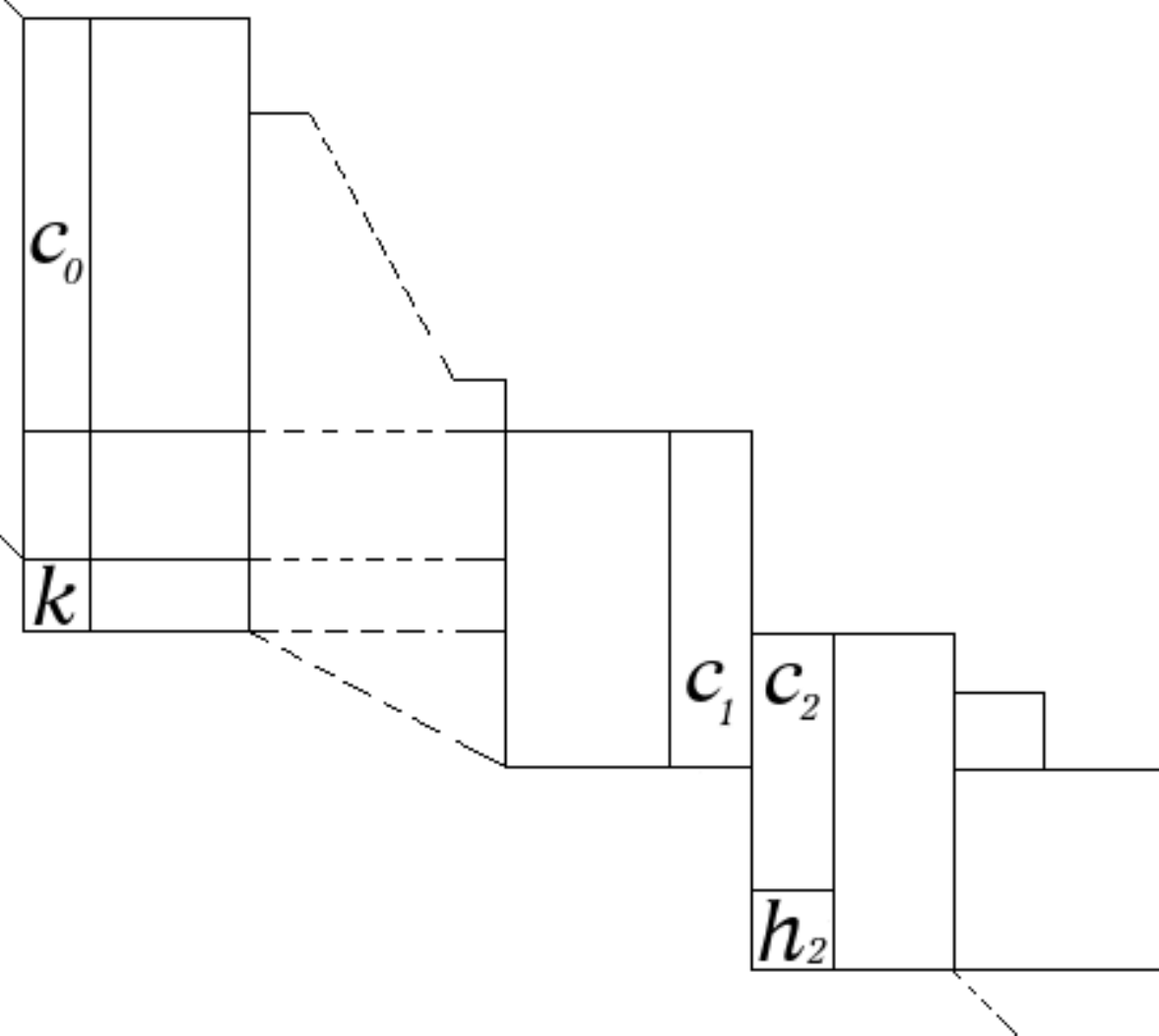}
\caption{Partial $k$-shape $s$.}
\label{figureexplicitation31}
\end{minipage}%
\begin{minipage}{.5\textwidth}
  \centering
\includegraphics[width=6cm]{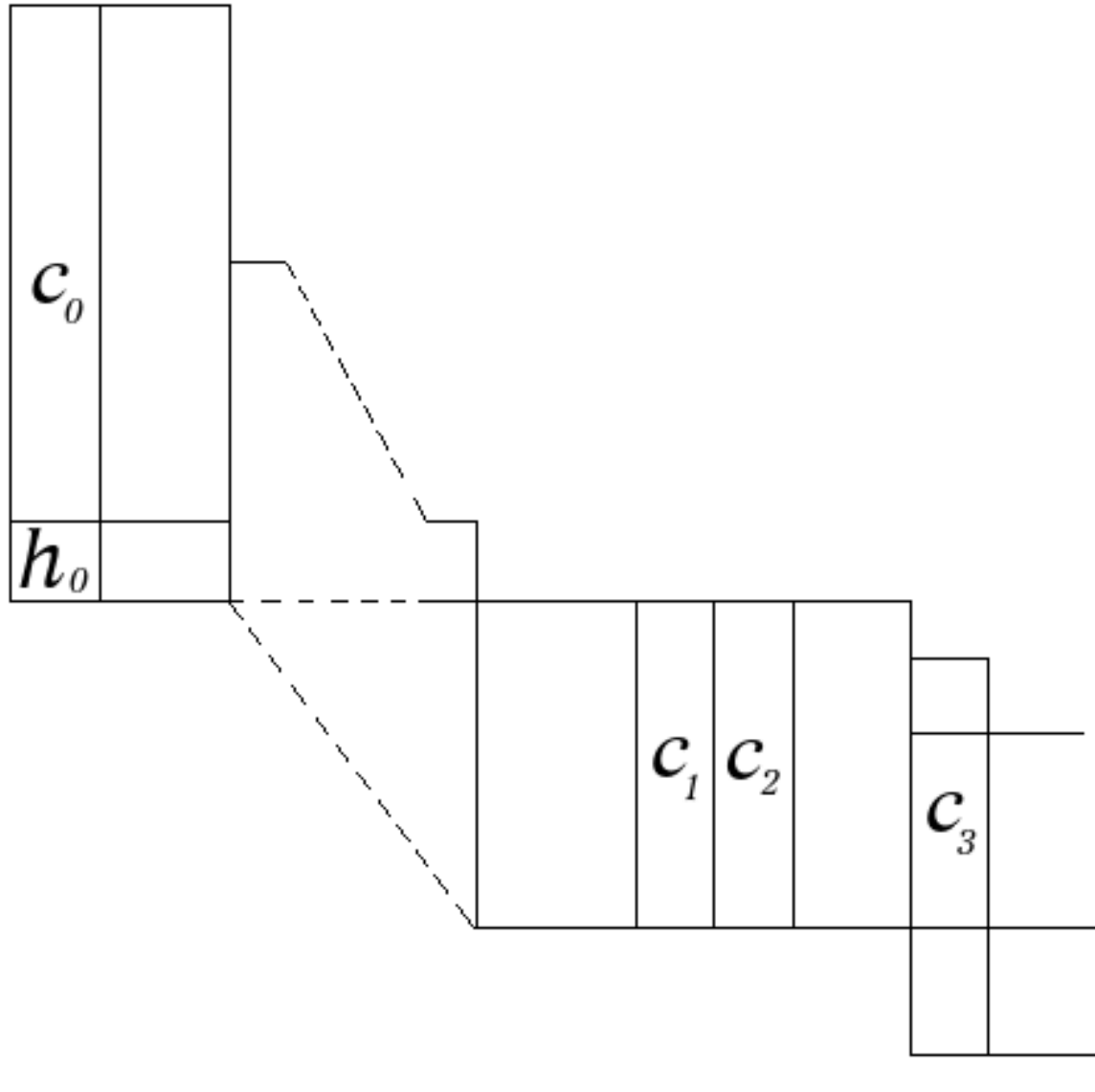}
\caption{Between $s$ and $s \oplus_{t(j)}^k \lceil (j+1)/2 \rceil^{z}$.}
\label{figureexplicitation32}
\end{minipage}%
\end{figure}
Since $C_0$ loses (momentarily) its saturation during the computation of $s \oplus_{t(j)}^k \lceil (j+1)/2 \rceil^{z}$, it is necessary that the columns $C_1$ and $C_2$ have the same height and the same label in order to obtain the situation depicted in Figure \ref{figureexplicitation32}.
Consequently, the lifting of $C_1$ in Figure \ref{figureexplicitation31} comes from rule $(1)$ of Definition \ref{sumskew} (it cannot be prompted by rule $(3)$ because $C_0$ has not lost its saturation yet). Also, if the label of $C_1$ and $C_2$ is $1$, then the hook length $h_2$ of the bottom cell $c_2$ of $C_2$ equals $k$ in Figure \ref{figureexplicitation32} (because $C_1$ has been lifted in the context $(1)$), implying the situation depicted in Figure \ref{figureexplicitation32} cannot be reached because, as noticed in Remark \ref{remarquesaturation}, the hook length $h_2$ of $c_2$ still equals $k$ in Figure \ref{figureexplicitation32}, forcing $C_1$ to be lifted by the rule $(1)$. So, the label of $C_1$ and $C_2$ must be $2$, and $h_2 = k-1$. \\
Finally, according to Lemma \ref{detaillifting1}, the lifting of $C_2$ in Figure \ref{figureexplicitation32} must be done in the context $(2)$ of Definition \ref{sumskew}: indeed, if it was context $(3)$, there would exist a column $C'_0$ labeled by $1$ between $C_0$ and $C_2$, which would be lifted so that its bottom cell ends up in the row $R_0$ (in order for $C_2$ to be lifted at the same level). But then, since $C_0$ has not lost its saturation yet at this time (because $C_2$ has not been lifted yet), the column $C'_0$ would be rooted in $R_0$ thus would be saturated, which is absurd because by hypothesis $C'_0$ is supposed to lose momentarily its saturation. Thus, the column $C_3$ depicted in Figure \ref{figureexplicitation32} is labeled by $1$.
\end{rem}
 
\begin{lem} \label{saturation}
Let $s$ be a partial $k$-shape and $j \geq 1$ such that the height of every column of $s$ is at least $\lceil (j+2)/2 \rceil$, and such that the quantity of integers $i \in [k-2]$ in which $s$ is not saturated is at most $\lceil j/2 \rceil$. Then, if $s$ is not saturated in $i_0 \in [k-2]$, there exists a unique integer $z \in [k-1-\lceil j/2 \rceil]$ such that the partial $k$-shape $s \oplus^k_{t(j)} \lceil (j+1)/2 \rceil^{z}$ is saturated in $i_0$.
\end{lem}

\begin{proof}
According to Remark \ref{explicitation3}, columns labeled by $1$ cannot be lifted in the context $(3)$ of Definition \ref{sumskew}. Consequently, in the partial $k$-shape $s$, the columns of height $i_0+1$ and label $1$ are organized in $m \geq 1$ groups of columns rooted in a same row, such that the $m-1$ first groups from right to left are made of saturated columns, \textit{i.e.}, such that the columns of these groups are rooted in rows whose top left cell is hook lengthed by $k$, and such that the $m$-th group is made of non-saturated columns, \textit{i.e.}, such that the columns $C_1,C_2,\hdots,C_q$ of this group (from left to right) are rooted in a row whose top left cell (which is the bottom cell $c_1$ of $C_1$) is hook lengthed  by some integer $h<k$ (see Figure \ref{exemplepropagation1}).
\begin{figure}[!h]
\centering
\begin{minipage}{.5\textwidth}
  \centering
\includegraphics[width=3.8cm]{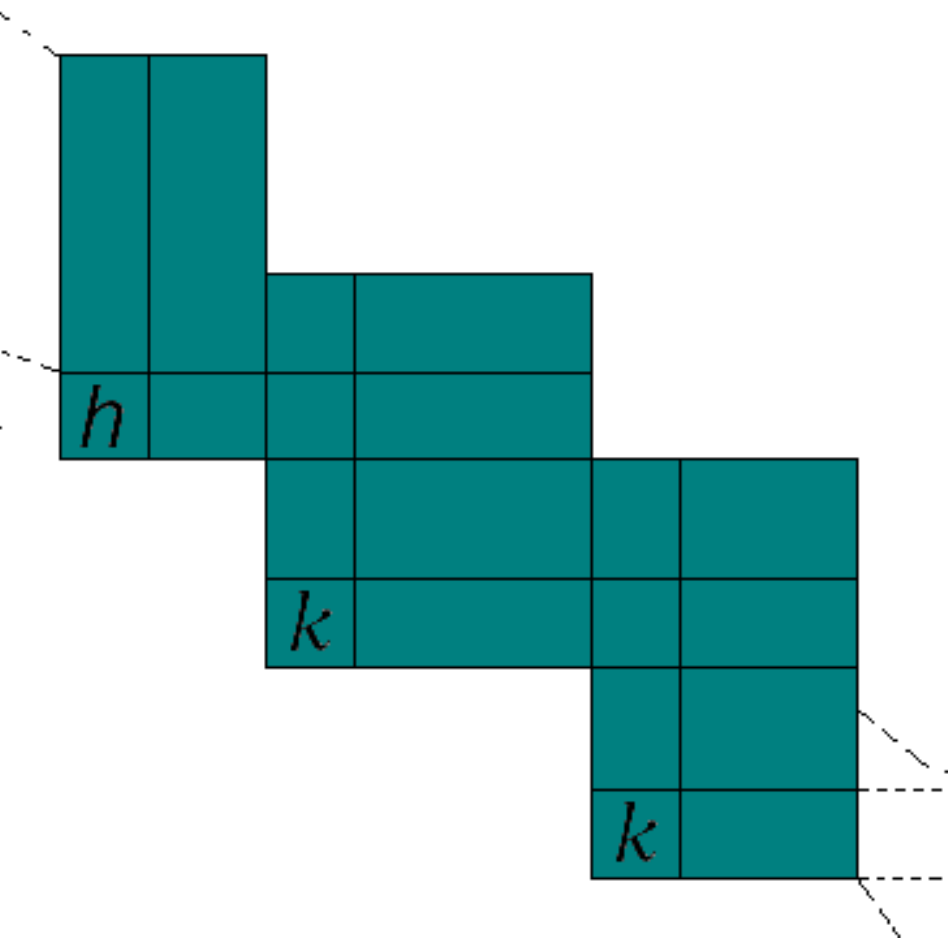}
\caption{Partial $k$-shape $s$.}
\label{exemplepropagation1}
\end{minipage}%
\begin{minipage}{.5\textwidth}
  \centering
\includegraphics[width=5cm]{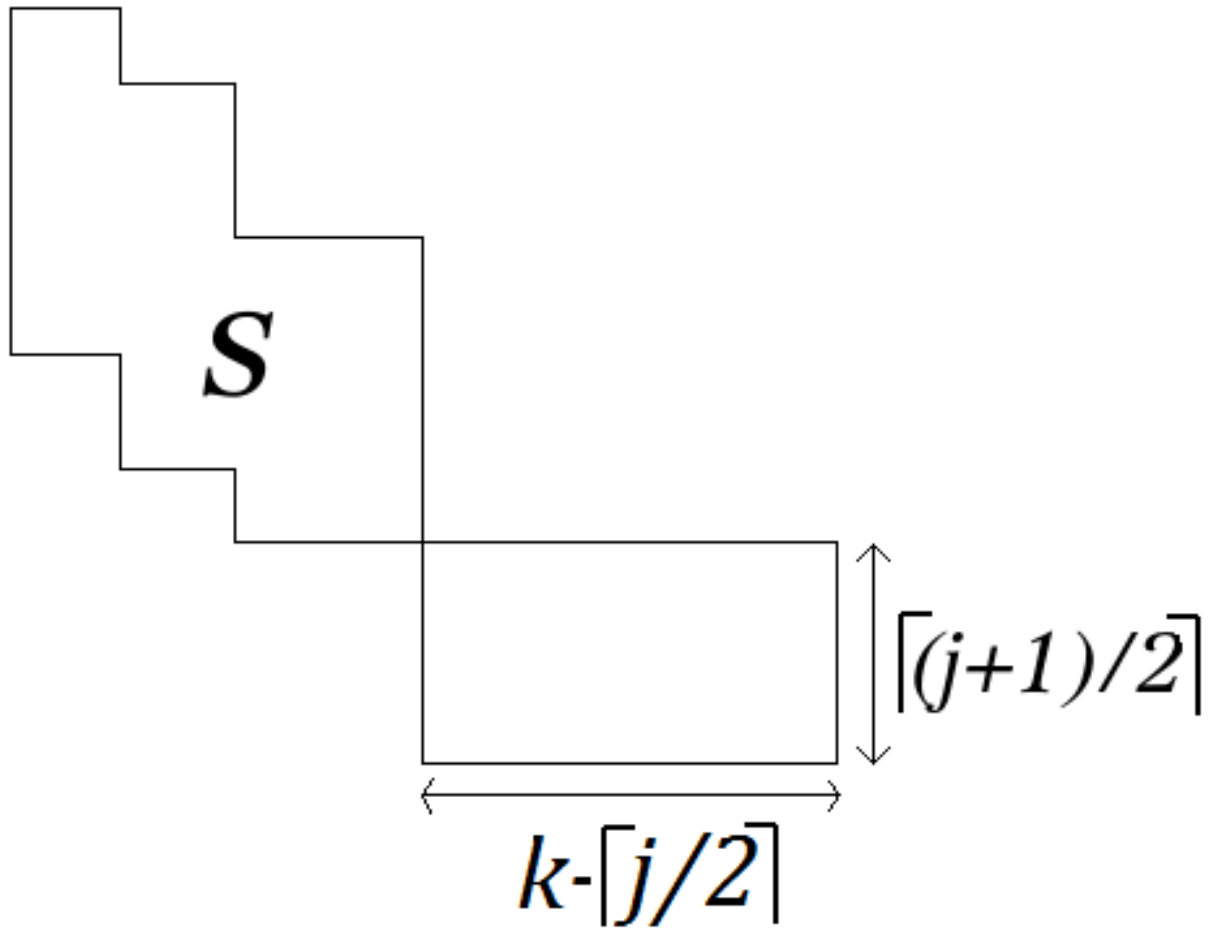}
\caption{Partial $k$-shape $s^{k-\lceil j/2 \rceil}$.}
\label{exemplepropagation2}
\end{minipage}
\end{figure}
\\
Now, for all $p \in [k-\lceil j/2 \rceil]$, let $s^p$ be the partial $k$-shape $s \oplus^k_{t(j)} \lceil (j+1)/2 \rceil^p$. For all $p \geq 2$, the partial $k$-shape $s^p$ is obtained by gluing a column of height $\lceil (j+1)/2 \rceil$ and label $t(j)$ right next to the last column of $s^{p-1}$, then by applying the 3 rules of Definition \ref{sumskew} so as to obtain a partial $k$-shape. We now focus on $s^{k-\lceil j/2 \rceil}$. Let $C$ be a column of $s$. Since the height of $C$ is at least $\lceil (j+2)/2 \rceil$, if the bottom cell $c$ of $C$ is located in the same row as one of the cells of the rectangle $\lceil (j+1)/2 \rceil^{k-\lceil j/2 \rceil}$ during the computation of $s^{k-\lceil j/2 \rceil}$, then the hook length of $c$ will be at least $\lceil (j+2)/2 \rceil + k - \lceil j/2 \rceil \geq k+1$. Consequently, according to the rule $(1)$ of Definition \ref{sumskew}, the column $C$ is lifted as long as its bottom cell is in the same row as one of the cell of the rectangle $\lceil (j+1)/2\rceil^{k-\lceil j/2 \rceil}$, and since this holds for every column $C$ of $s$, then the partial $k$-shape $s^{k-\lceil j/2 \rceil}$ is obtained by drawing the rectangle $\lceil (j+1)/2 \rceil^{k-\lceil j/2 \rceil}$ in the bottom right hand corner of $s$ (see Figure \ref{exemplepropagation2}). In particular, the columns $C_1,C_2,\hdots,C_q$ must have been been lifted $\lceil (j+1)/2 \rceil$ times between $s$ and $s^{k-\lceil j/2 \rceil}$. The idea is to prove there exists a unique $p_0 \in [k-1-\lceil j/2 \rceil]$ such that $C_1$ has been lifted in the context $(1)$ of Definition \ref{sumskew} in $s^{p_0+1}$, implying the hook length $h$ of $c_1$ equals $k$ in $s^{p_0}$, which means $s^{p_0}$ is saturated in $i_0$ according to Remark \ref{remarquesaturation}. Note that the columns $C_1,C_2,\hdots,C_q$ cannot be lifted in the context $(3)$ of Definition \ref{sumskew} because their label is $1$ (see Remark \ref{explicitation3}). If $m \geq 2$, the columns $C_1,C_2,\hdots,C_q$ cannot be lifted in the context $(2)$, hence they are lifted in the context $(1)$, so the existence and unicity of the integer $p_0$ is obvious. If $m = 1$, suppose $C_1$ is never lifted in the context $(1)$, \textit{i.e.}, that it is lifted $\lceil (j+1)/2 \rceil$ times in the context $(2)$ between $s$ and $s^{k-\lceil j/2 \rceil}$. Each of these $\lceil (j+1)/2 \rceil$ times, the first column labeled by $1$ (from right to left) prompting the chain reaction of liftings in the context $(2)$ (which leads to the lifting of $C_1$), must be different from the others. Also, the columns labeled by $1$ responsible for these liftings cannot be saturated, because from Remark \ref{remarquesaturation} they would still be saturated when glued to the columns of different height or label that are lifted, meaning their bottom cell would be hooked lengthed by $k$ and that the lifting would be in the context $(1)$, which is absurd. As a conclusion, it is necessary that in addition to $i_0$, there would exist $\lceil (j+1)/2 \rceil \geq \lceil j/2 \rceil$ different integers $i < i_0$ such that $s$ is not saturated in $i$, which is absurd by hypothesis.
\end{proof}

\section{Proof of Theorem \ref{theoremebijection}}
\label{sec:proof}

We first construct two key algorithms in the first two subsections.

\subsection{Algorithm $\varphi : SP_{k-1} \rightarrow IS_k$}
\label{sec:varphi}

\begin{defi}[Algorithm $\varphi$] \label{definitionvarphi}
Let $f \in SP_{k-1}$. We define $s^{2k-3}(f)$ as the empty skew partition. For $j$ from $2k-4$ downto $1$, let $i \in [k-1]$ such that $f(j) = 2i$, and suppose that the hypothesis $H(j+1)$ defined as "if $s^{j+1}(f)$ is not empty, the height of every column of $s^{j+1}(f)$ is at least $\lceil (j+2)/2 \rceil$, and the number of integers $i$ in which $s^{j+1}(f)$ is not saturated is at most $\lceil j/2 \rceil$" is true (in particular $H(2k-3)$ is true so we can initiate the algorithm).
\begin{enumerate}
\item If $f(2i) > 2i$, if $ j = \min \{j' \in [2k-4],f(j') = 2i \}$ and if the partial $k$-shape $s^{j+1}$ is not saturated in $i$, then we define $s^j(f)$ as $s^{j+1}(f) \oplus_{t(j)}^k \lceil (j+1)/2 \rceil^{z_j(f)}$ where $z_j(f)$ is the unique element of $[k-1-\lceil j/2 \rceil]$ such that $s^{j+1}(f) \oplus_{t(j)}^k \lceil (j+1)/2 \rceil^{z_j(f)}$ is saturated in $i$ (see Lemma \ref{saturation} in view of Hypothesis $H(j+1)$).
\item Else, we define $s^j(f)$ as $s^{j+1}(f) \oplus_{t(j)}^k \lceil (j+1)/2 \rceil^{z_j(f)}$ where $f(j) = 2 (\lceil j/2 \rceil + z_j(f))$ (notice that $z_j(f) \in \{0,1,\hdots,k-1-\lceil j/2 \rceil \}$ by definition of a surjective pistol).
\end{enumerate}
In either case, if $s^j(f)$ is not empty, then the height of every column is at least $\lceil (j+1)/2 \rceil$. Also, suppose there exists at least $\lceil (j-1)/2 \rceil+1$ different integers $i \in [k-2]$ in which $s^j(f)$ is not saturated. In view of the rule $(1)$ of the present algorithm, this implies there are at least $\lceil (j-1)/2 \rceil+1$ integers $j'\leq j-1$ such that $f(j') \geq 2 \lceil j/2 \rceil$. Also, since $f$ is surjective, there exist at least $\lceil j/2 \rceil-1$ integers $j'' \leq j-1$ such that $f(j') \leq 2(\lceil j/2 \rceil -1)$. Consequently, we obtain $\left(\lceil (j-1)/2 \rceil + 1\right) + \left(\lceil j/2 \rceil -1\right) \leq j-1$, which cannot be because $\lceil (j-1)/2 \rceil + \lceil j/2 \rceil = j$. So the hypothesis $H(j)$ is true and the algorithm goes on. Ultimately, we define $\varphi(f)$ as the partition $<s^1(f)>$.
\end{defi}

\begin{prop} \label{varphikshape}
For all $f \in SP_{k-1}$, the partition $\lambda = \varphi(f)$ is an irreducible $k$-shape such that $\partial^k(\lambda) = s^1(f)$ and $\overrightarrow{fr}(\lambda) = \overrightarrow{fix}(f)$.
\end{prop}
\hspace*{-5.8mm}
For example, consider the surjective pistol $f=(2,8,4,10,10,6,8,10,10,10) \in SP_5$ whose tableau is depicted in Figure \ref{exemplepistol}. Apart from $10$, the only fixed point of $f$ is $6$, so $\overrightarrow{fix}(f) = (0,0,1,0)$.
\begin{figure}[!h]
\centering
\includegraphics[width=4cm]{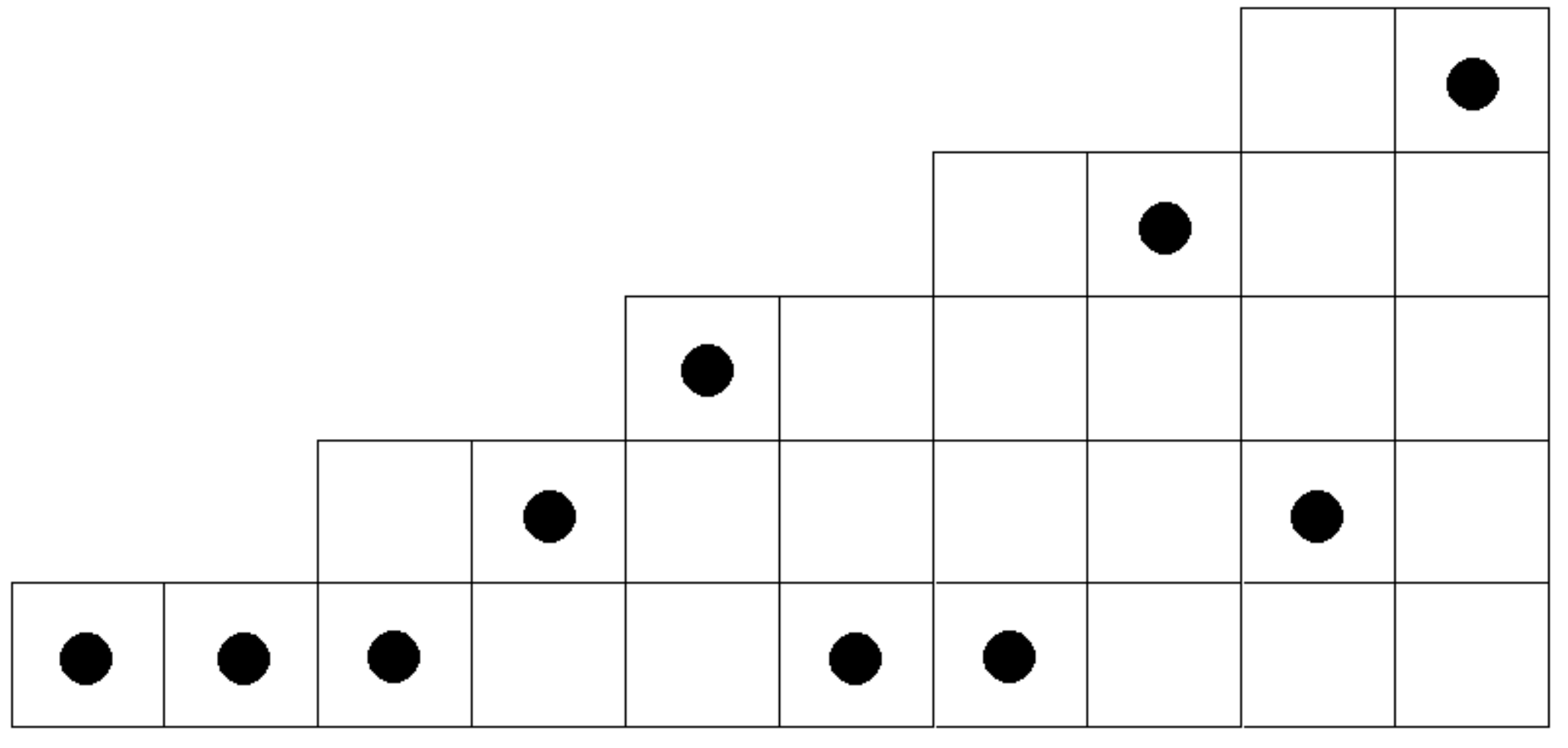}
\caption{Surjective pistol $f = (2,8,4,10,10,6,8,10,10,10) \in SP_5$.}
\label{exemplepistol}
\end{figure}
\\
Algorithm $\varphi$ provides the sequence $(s^8(f),s^7(f),\hdots,s^1(f))$ depicted in Figure \ref{tableau} (note that $s^8(f) = s^7(f) = s^6(f)$ because $z_7(f) = z_6(f) = 0$).
\begin{figure}[!h]
\begin{tabular}{|c|c|c|c|c|c|}
\hline
 $j \in \{8;7;6\}$ & $j=5$ & $j=4$ & $j=3$ & $j=2$ & $j=1$\\
 \hline
 & & & & &\\
\includegraphics[width=0.5cm]{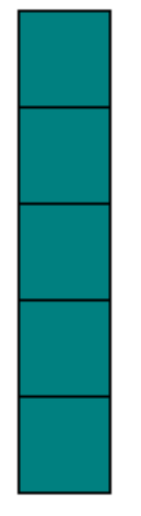} & \includegraphics[width=1cm]{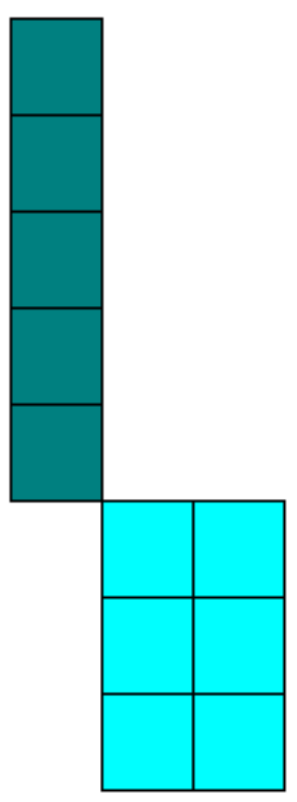} & \includegraphics[width=2.05cm]{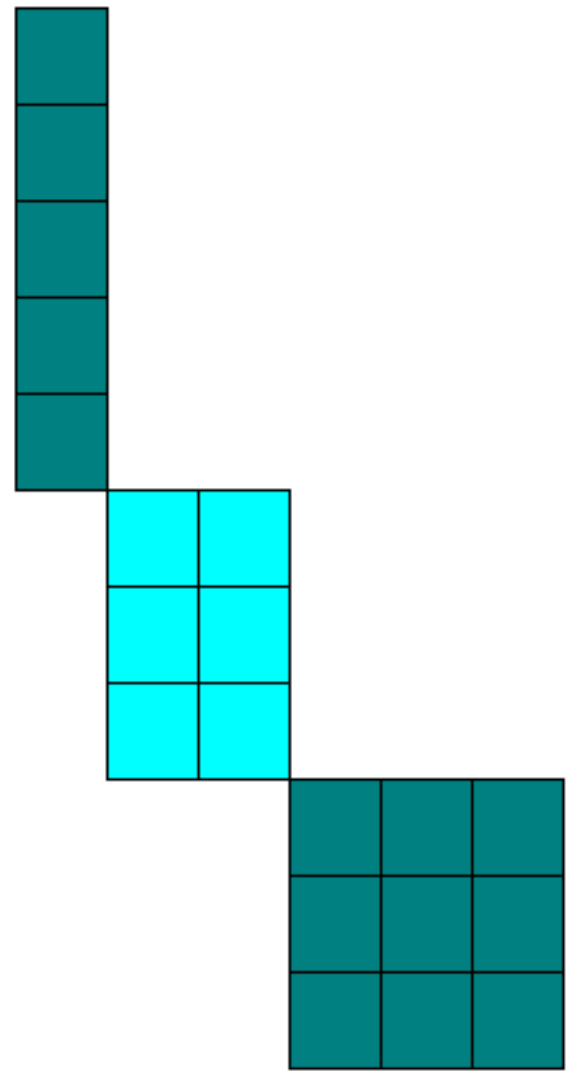} & \includegraphics[width=2.3cm]{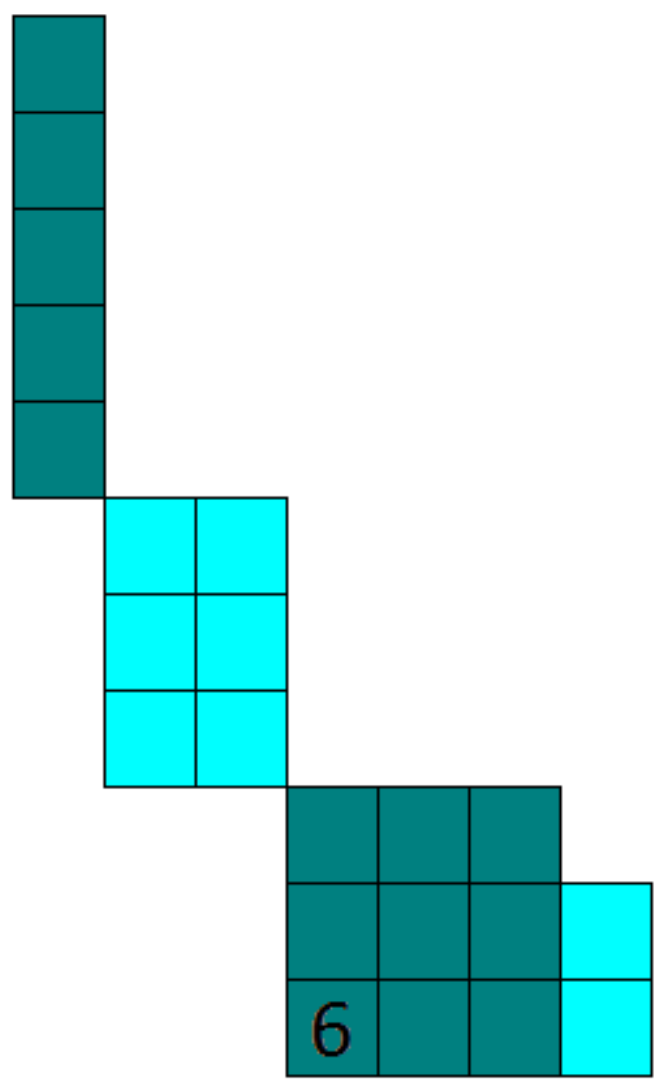} & \includegraphics[width=3cm]{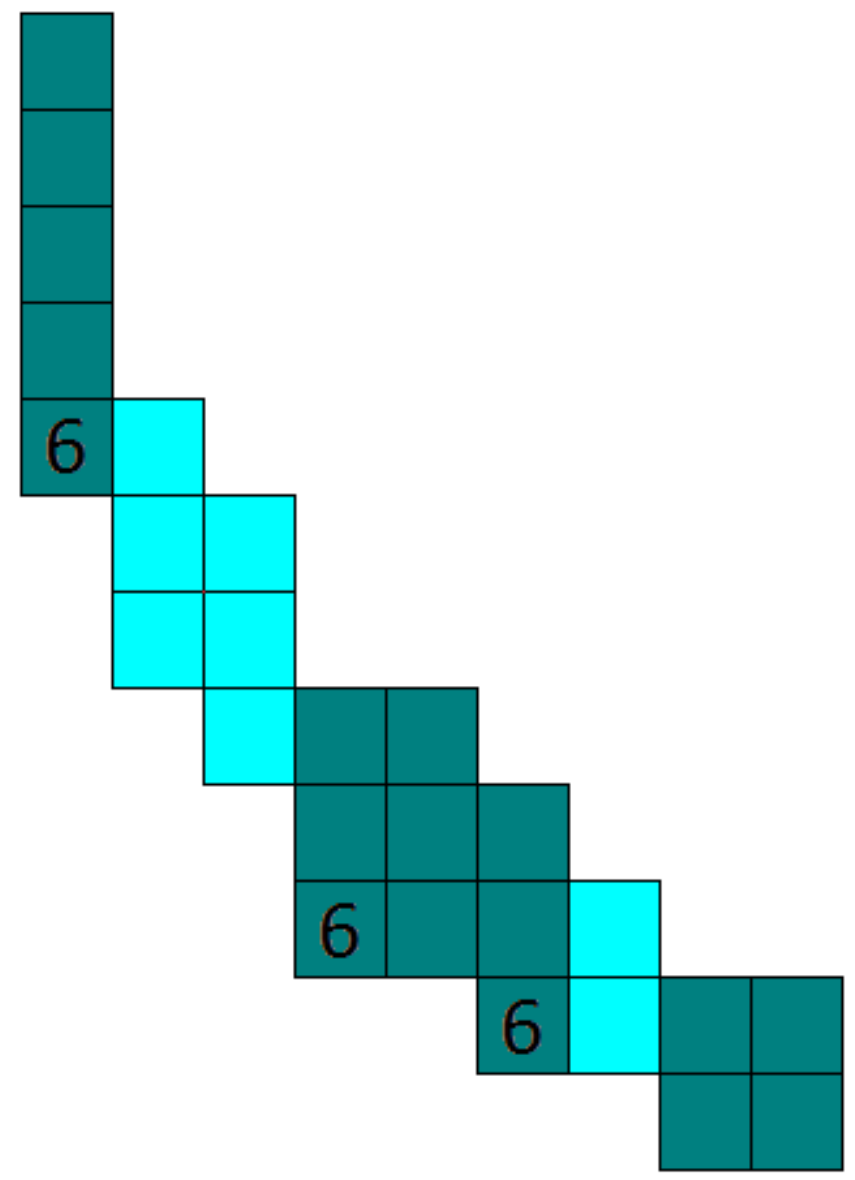} & \includegraphics[width=4cm]{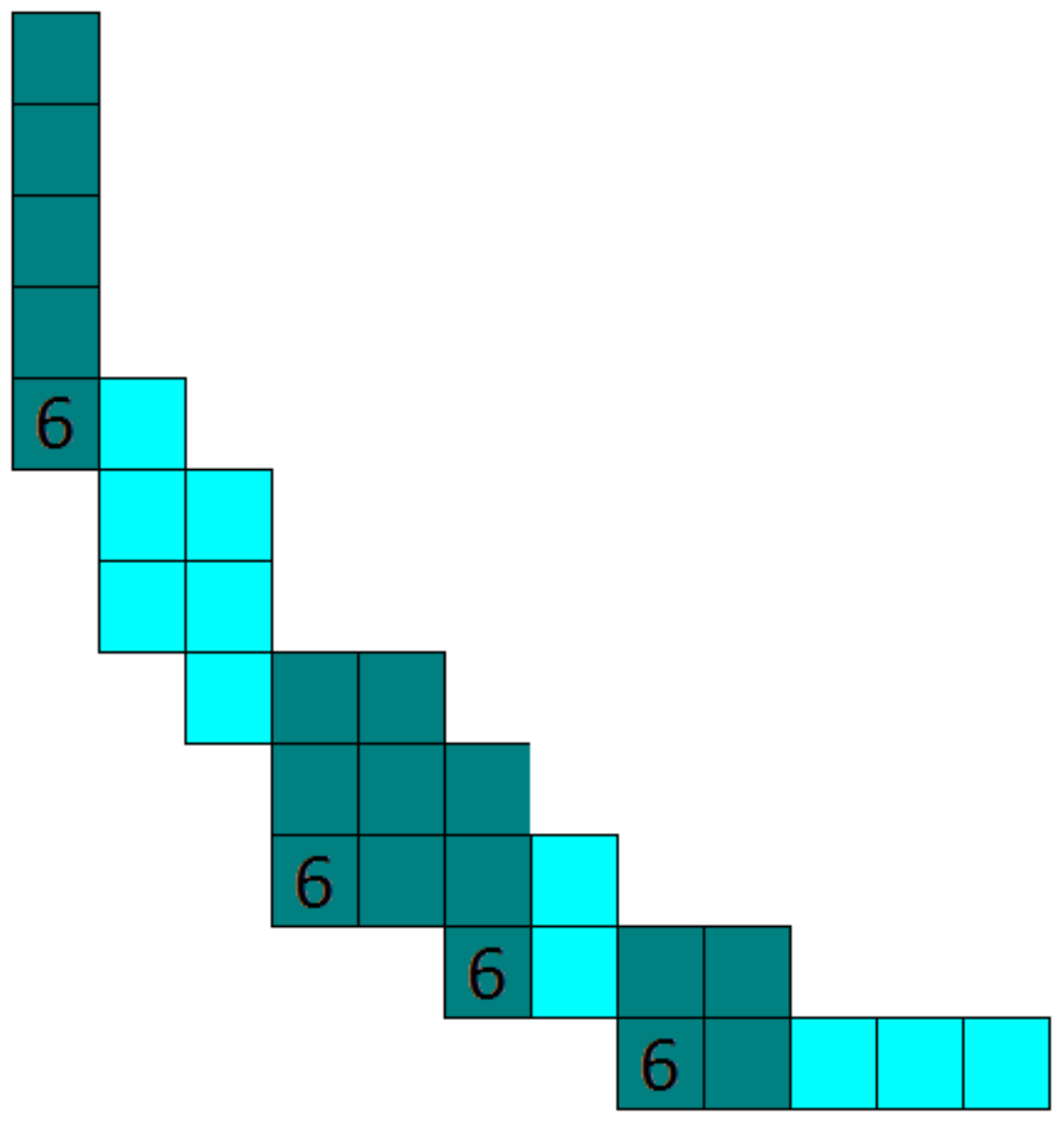} \\
 \hline
\end{tabular}
 \caption{Sequence $(s^j(f))_{j \in [8]}$.}
 \label{tableau}
\end{figure}
\\
Thus, we obtain $s^1(f) = \partial^6(\lambda)$ where $\lambda = \varphi(f) = <s^1(f)>$. In particular, the sequences $rs^6(\lambda) = (5,4,4,3,\hdots,1)$ and $cs^6(\lambda) = (5,3,3,3,\hdots,1)$ are partitions, so $\lambda$ is a $6$-shape. Finally, we can see in Figure \ref{exempleirreducible} that $\lambda$ is irreducible and $\overrightarrow{fr}(\lambda) = (0,0,1,0) = \overrightarrow{fix}(f)$.
\begin{figure}[!h]
\centering
\includegraphics[width=5cm]{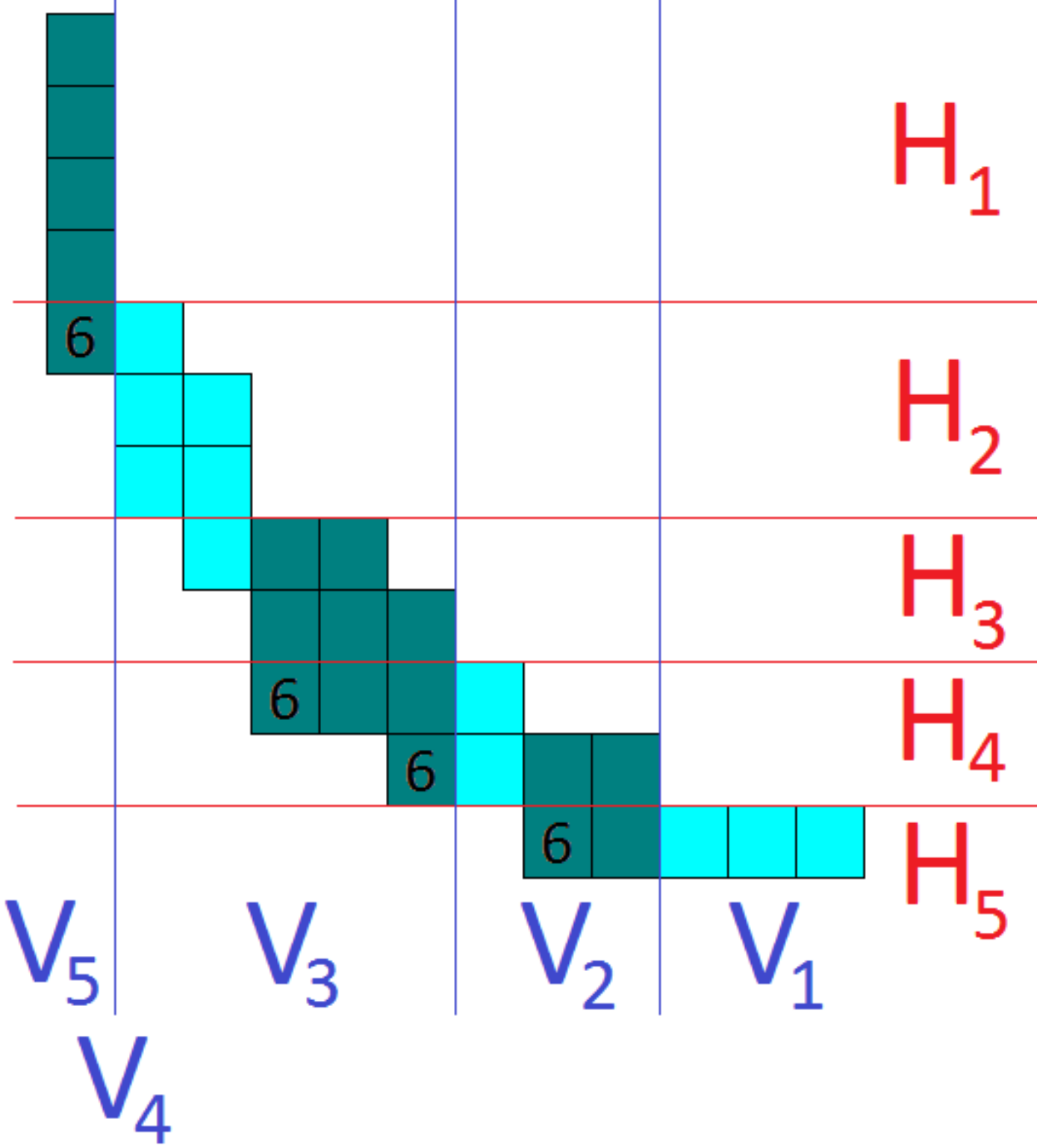}
\caption{$6$-boundary $\partial^6(\lambda) = s^1(f)$ of the irreducible $6$-shape $\lambda = \varphi(f)$.}
\label{exempleirreducible}
\end{figure}
\\
We split the proof of Proposition \ref{varphikshape} into Lemmas \ref{lemmapartialk}, \ref{lemmars} and \ref{lemmairreducible}.

\begin{lem} \label{lemmapartialk} For all $f \in SP_{k-1}$, we have $\partial^k(\varphi(f)) = s^1(f)$.
\end{lem}
\hspace*{-6mm} \textbf{Proof.}
By construction, the skew partition $s^1(f)$ is a saturated partial $k$-shape (the saturation is guaranteed by Hypothesis $H(1)$). As a partial $k$-shape, the hook length of every cell doesn't exceed $k$. Consequently, to prove that $\partial^k(\varphi(f)) = s^1(f)$, we only need to show that the hook length $h_1$ of every \textit{anticoin} of $s^1(f)$ (namely, cells glued simultaneously to the left of a row of $s^1(f)$ and beneath a column of $s^1(f)$, see Figure \ref{anticoin}) is such that $h_1 > k$.
\begin{figure}[!h]
\centering
\begin{minipage}{.5\textwidth}
  \centering
\includegraphics[width=3cm]{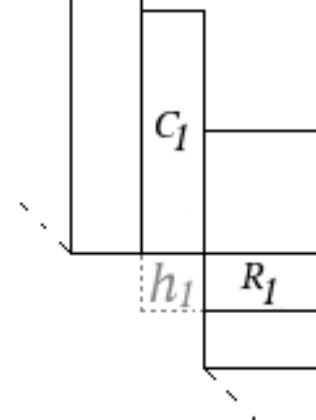}
\caption{Anticoin labeled by its hook length $h_1$.}
\label{anticoin}
\end{minipage}%
\begin{minipage}{.5\textwidth}
  \centering
\includegraphics[width=3cm]{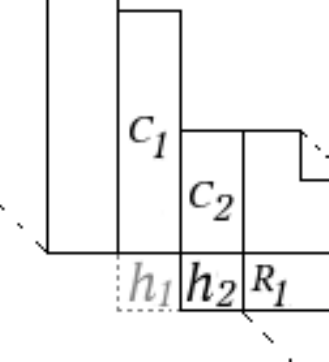}
\caption{}
\label{anticoin2}
\end{minipage}
\end{figure}
\\
Anticoins of $s^1(f)$ are created by lifting columns in one of the three contexts $(1)$,$(2)$ or $(3)$ of Definition \ref{sumskew}. Let $x_1$ (resp. $y_1$) be the length of the row $R_1$ (resp. the height of the column $C_1$).
\begin{enumerate}
\item If $C_1$ has been lifted in the context $(1)$, then $x_1+y_1 > k$ (if $c_1$ is labeled by $1$) or $x_1+y_1 > k-1$ (if $C_2$ is labeled by $2$). In either case, we obtain $h_1 = 1+x_1+y_1 > k$.
\item If $C_1$ has been lifed in the context $(2)$, then the first cell (from left to right) of the row $R_1$ is a corner, and it is the bottom cell of a column $C_2$ labeled by $1$ (see Figure \ref{anticoin2}). Let $y_2$ be the height of $C_2$. Since $C_2$ is saturated, then the hook length $h_2 = x_1+y_2-1$ of its bottom cell equals $k$. Consequently, since $y_1 \geq y_2$, we obtain $h_1 = x_1+y_1-1 > k$.
\item Else $C_1$ has been lifted in the context $(3)$. Let $C_0$ be the saturated column of $s^1(f)$ such that $C_1$ is the column that contains the last cell (from left to right) of the row $R_0$ in which $C_0$ is rooted (see Figure \ref{anticoin3}).
\begin{figure}[!h]
\centering
\begin{minipage}{.35\textwidth}
  \centering
\includegraphics[width=4cm]{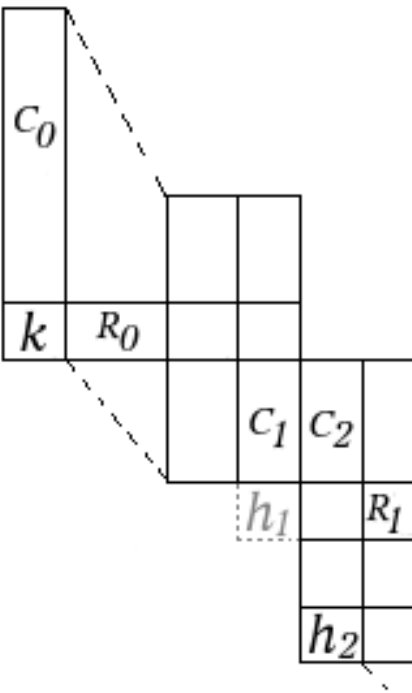}
\caption{$s^1(f)$.}
\label{anticoin3}
\end{minipage}%
\begin{minipage}{.35\textwidth}
  \centering
\includegraphics[width=8cm]{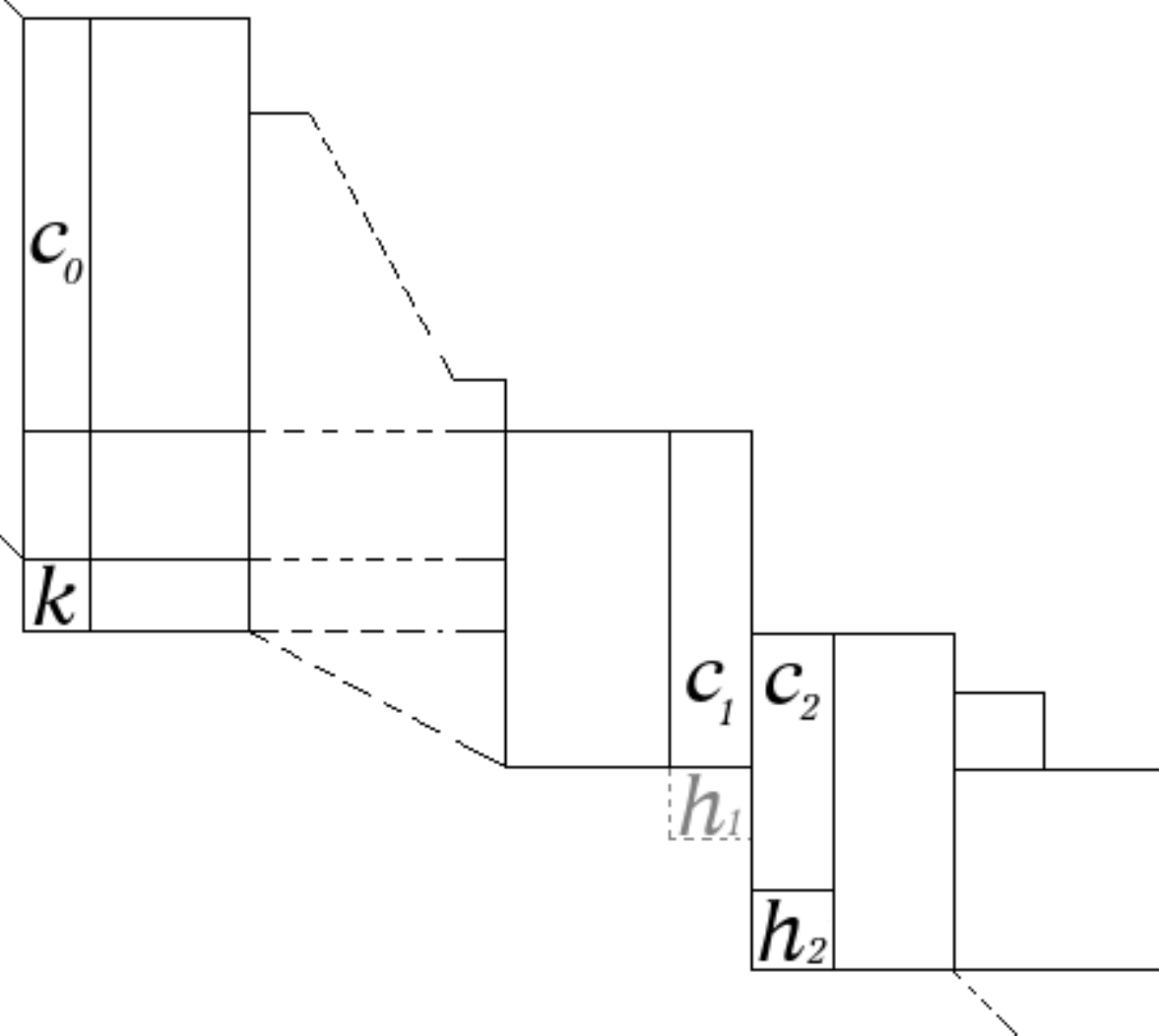}
\caption{$s^{j+1}(f)$.}
\label{democonservation1}
\end{minipage}
\end{figure}
\\
Let $j \in [2k-4]$ such that $C_0$ is saturated in the partial $k$-shape $s^{j+1}(f)$ and such that $C_0$ loses its saturation at some point of the computation of $s^j(f)$ (recall that $C_0$ is saturated at the end of this computation by the rule $(3)$ of Definition \ref{sumskew}). The partial $k$-shape $s^{j+1}(f)$ then presents the situation depicted in Figure \ref{democonservation1}.
Following Remark \ref{explicitation3}, the columns $C_1$ and $C_2$ have the same height and the same label $2$ (and $h_2 = k-1$), and in order for $C_0$ to lose temporarily its saturation between $s^{j+1}(f)$ and $s^j(f)$, there exists a column $C_3$ labeled by $1$ on the right of $C_2$ such that $C_2$ is lifted at the same level as $C_1$ in the context $(2)$ of Definition \ref{sumskew}, with $C_3$ being the column responsible for this lifting (see Figure \ref{democonservation3}).
\begin{figure}[!h]
\centering
\begin{minipage}{.5\textwidth}
  \centering
\includegraphics[width=9cm]{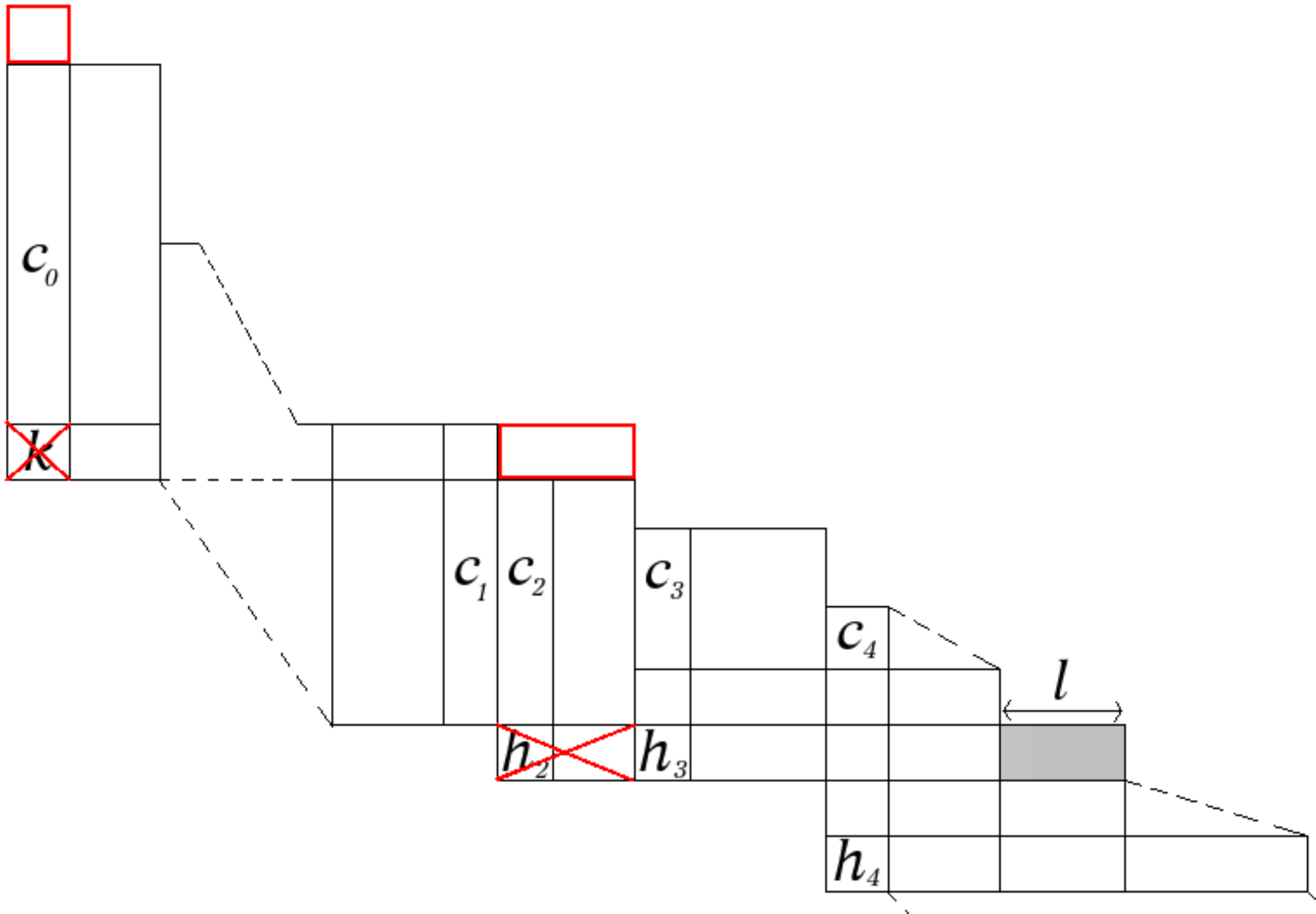}
\caption{Between $s^{j+1}(f)$ and $s^j(f)$.}
\label{democonservation3}
\end{minipage}%
\begin{minipage}{.5\textwidth}
  \centering
\includegraphics[width=9cm]{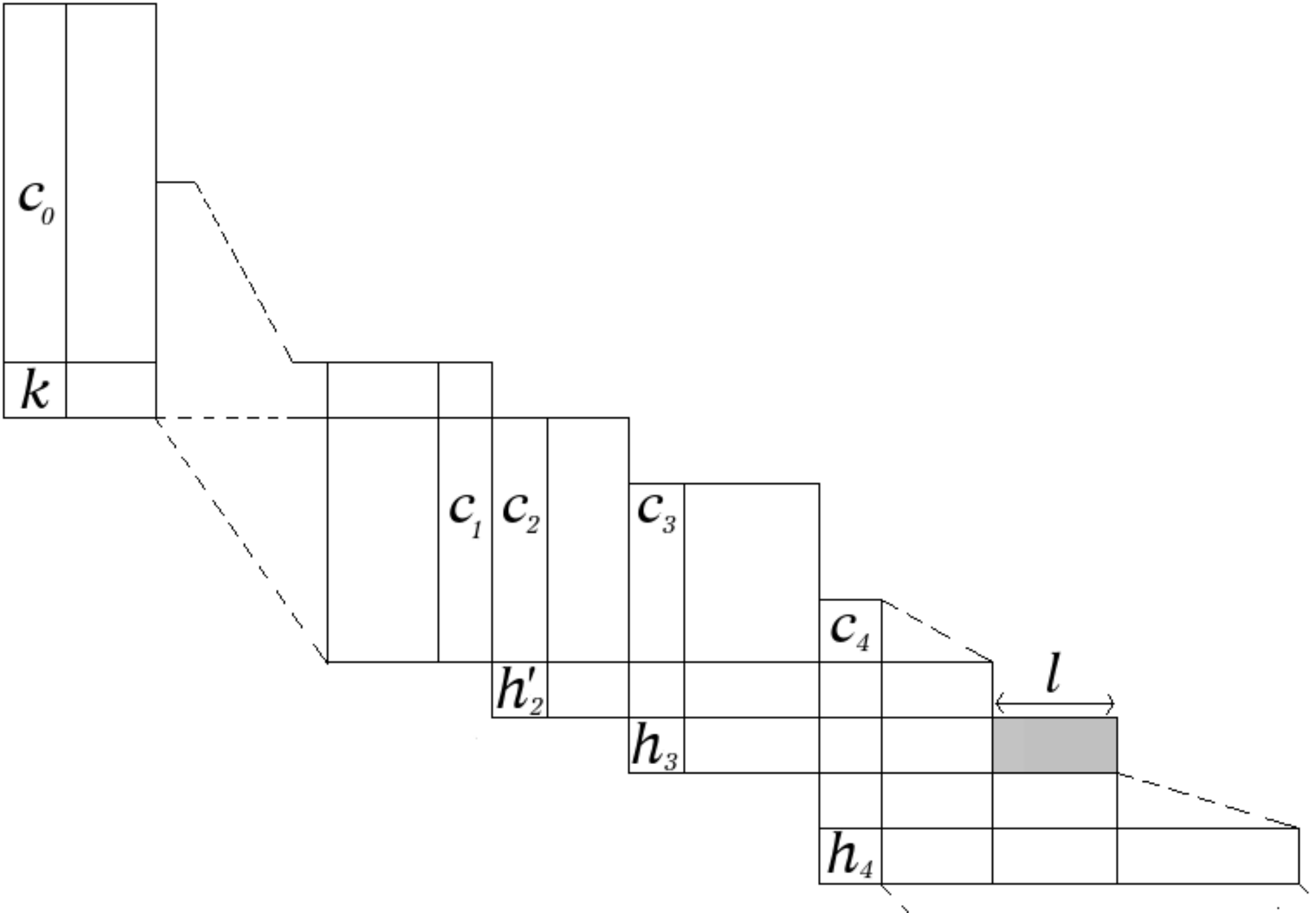}
\caption{$s^{j}(f)$}
\label{democonservation31}
\end{minipage}
\end{figure}
\\
Now, the rule $(3)$ of Definition \ref{sumskew} yields the situation depicted in Figure \ref{democonservation31}, in which $h'_2 = h_2 - l$ where $l \geq 1$ is the number of gray cells.\\
Finally, in order to saturate $C_3$ (every column labeled by $1$ being saturated in $s^1(f)$), it is necessary to lift the $l$ gray cells. Indeed, otherwise, the column $C_3$ would become saturated by lifting columns $C_5,C_6,\hdots,C_m$ whose top cells would be glued to the right of the last gray cells, meaning those columns have the same height $y$ and the same label $t$ as the $l$ columns whose top cells are the gray cells. Obviously, since $C_3$ is not saturated yet at that moment, the columns $C_5,C_6,\hdots,C_m$ are not lifted in the context $(3)$ of Definition \ref{sumskew}. In view of Lemma \ref{detaillifting1}, these columns must have been lifted in the context $(2)$, which implies every column of height $y$ and label $t$ are lifted to the same level as the last gray cell. But then it would lift the column $C_3$ in the context $(1)$ of Definition \ref{sumskew} instead of saturating it, which is absurd. So the $l$ gray cells are necessarily lifted in $s^1(f)$ (see Figure \ref{democonservation32}), in which $h_2'$ has become $h_2$ again).
\begin{figure}[!h]
\centering
\begin{minipage}{.5\textwidth}
  \centering
\includegraphics[width=9cm]{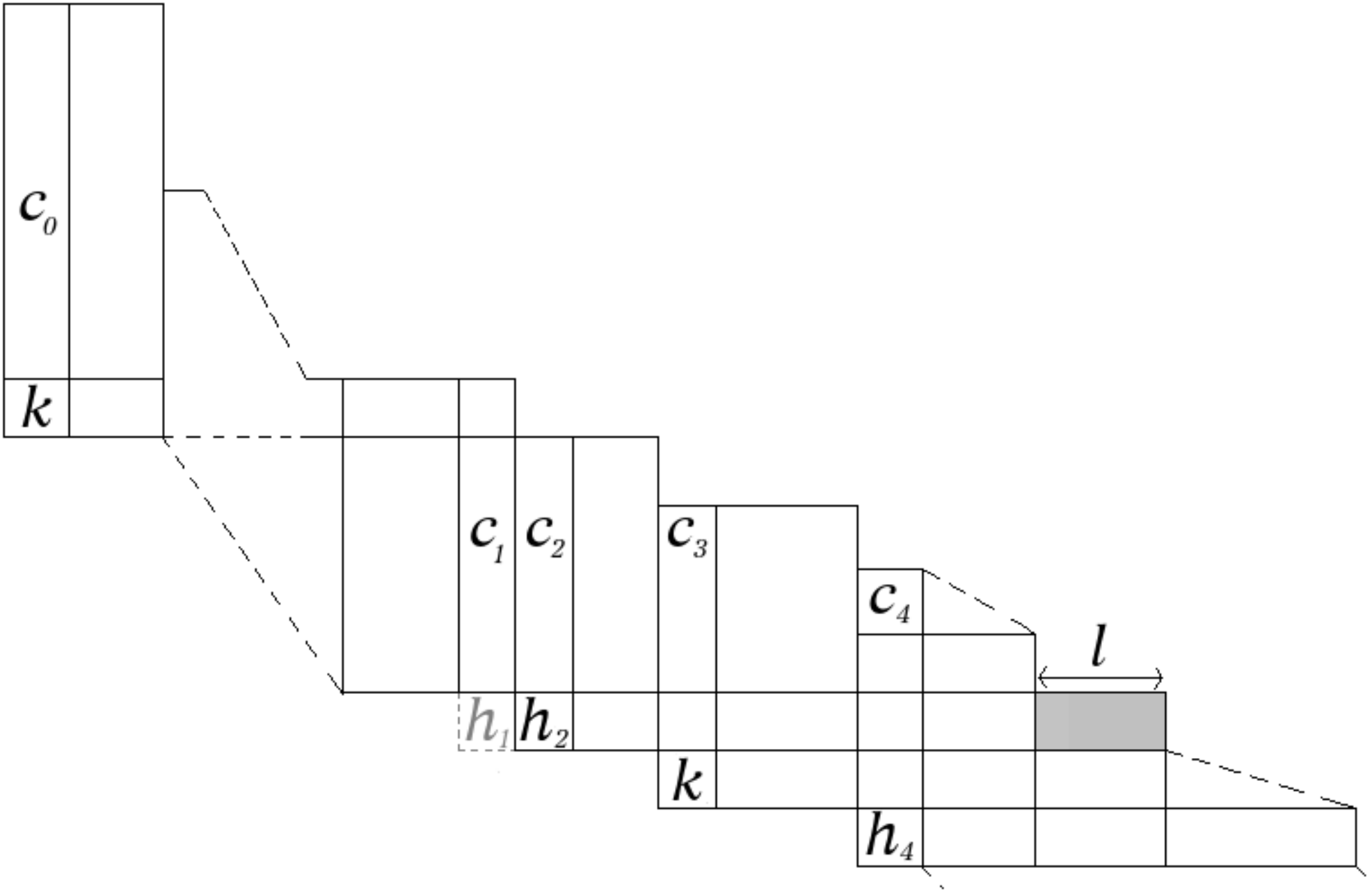}
\caption{$s^{1}(f)$}
\label{democonservation32}
\end{minipage}%
\begin{minipage}{.5\textwidth}
  \centering
\includegraphics[width=10cm]{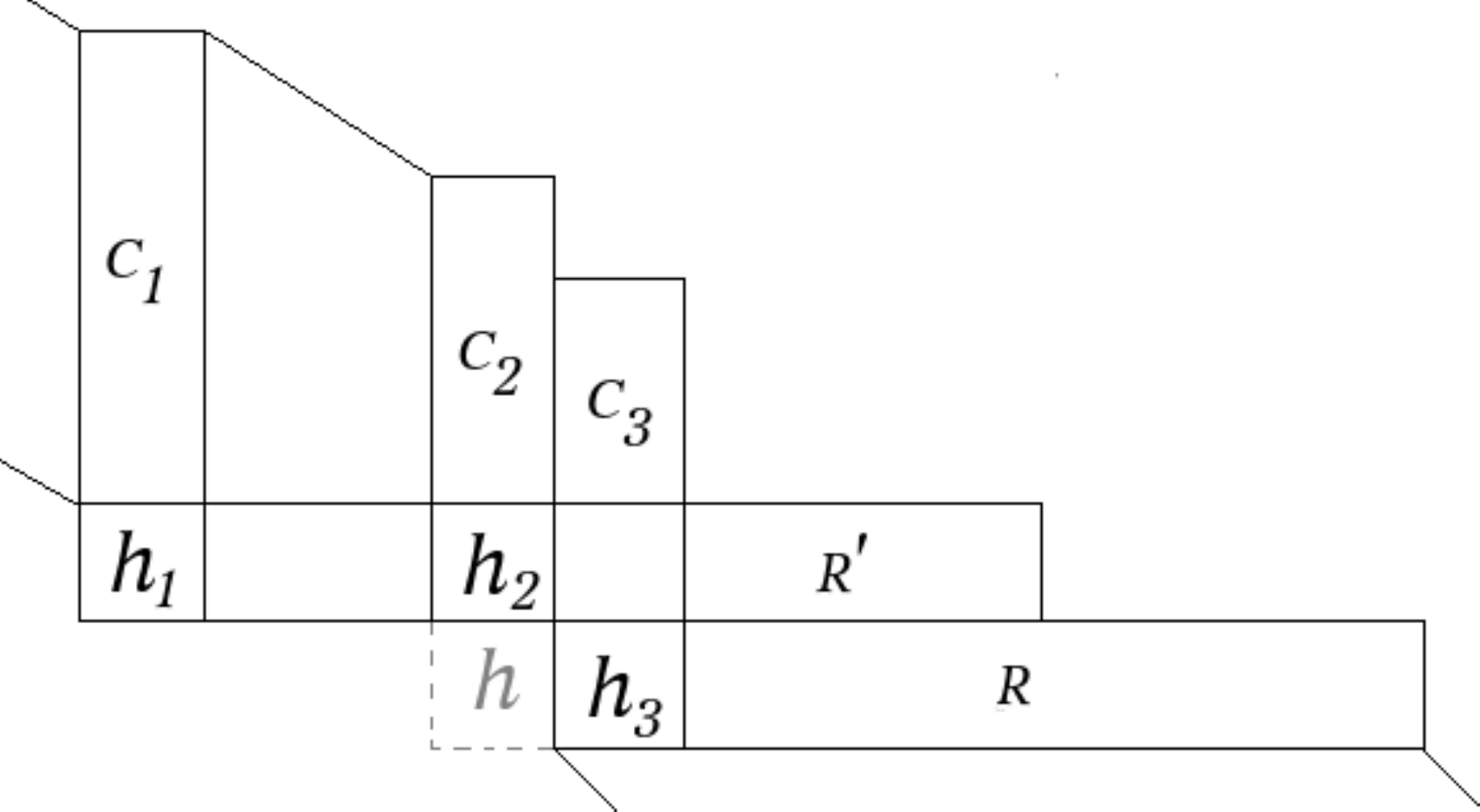}
\caption{$s^{1}$}
\label{preuvers}
\end{minipage}
\end{figure}
\\
In particular, the hook length $h'_2 = h_2$ equals $k-1$ and we obtain $h_1 = h_2+ 2 > k$.\qed
\end{enumerate}

\begin{lem} \label{lemmars}
For all $f \in SP_{k-1}$, the partition $\lambda = \varphi(f)$ is a $k$-shape.
\end{lem}

\begin{proof}
From Lemma \ref{lemmapartialk} we know that $s^1(f) = \partial^k(\lambda)$, and since $s^1(f)$ is a partial $k$-shape by construction, the sequence $cs^k(\lambda) = cs(s^1(f))$ is a partition. To prove that $\lambda$ is a $k$-shape, it remains to show that the sequence $rs^k(\lambda) = rs(s^1(f))$ is a partition. Let $R$ and $R'$ be two consecutive rows (from bottom to top) of $s^1(f)$ (see Figure \ref{preuvers}).
Let $x,x'$ be the respective lengths of $R,R'$ and $y_1 \geq y_2 \geq y_3$ the respective heights of the columns $C_1,C_2,C_3$ introduced by the picture. Since $s^1(f) = \partial^k(\lambda)$, we know that the quantity $h = y_2+x+1$ (which is the hook length of the cell glued to $C_2$ and $R$) exceeds $k$. Also, the hook length $h_1$ of the bottom cell $c_1$ of $C_1$ doesn't exceed $k$, so $x' = h_1-y_1+1 \leq k-y_1+1 \leq k-y_2+1 \leq x+1$. Now suppose that $x' = x+1$. Then $h_1 = y_1+x'-1 = y_1+x  \geq y_2 + x = h - 1 \geq k$, so $h_1 = k$ (which implies $y_1 = y_2$). Consequently $C_1$ and $C_2$ are two columns of height $y_1 = y_2$, and labeled by $1$ because $h_1 = k$. Also, we have $h_3 = y_3 + x - 1 = y_3 + x'-2 \leq y_1+x'-2 = k-1$. Since every column labeled by $1$ in $s^1(f)$ is saturated, it forces $C_3$ to be labeled by $2$, which implies the column $C_2$ has not been lifted in the context $(2)$ of Definition \ref{sumskew}. Now $y_2 + x = y_1 + x'-1 = k$ so $C_2$ has not been lifted in the context $(1)$ of Definition \ref{sumskew} either. Consequently $C_2$ has been lifted in the context $(3)$. Now, following Remark \ref{explicitation3}, it implies $C_2$ and $C_3$ have the same height, \textit{i.e.}, that $y_2 = y_3$. This is impossible because, by construction of $\varphi(f)$, columns of height $y_2 = y_3$ labeled by $2$ (including $C_3$) are positionned before columns of height $y_2$ labeled by $1$ (including $C_2$). As a conclusion, it is necessary that $x \geq x'$, then $rs^k(\lambda)$ is a partition and $\lambda$ is a $k$-shape.
\end{proof}

\begin{lem} \label{lemmairreducible}
For all $f \in SP_{k-1}$, the $k$-shape $\lambda = \varphi(f)$ is irreducible and $\overrightarrow{fr}(\lambda) = \overrightarrow{fix}(f)$.
\end{lem}

\begin{proof}
For all $i \in [k-2]$, let $n_i$ (resp. $m_i$) be the number of horizontal steps of the $k$-rim of $\lambda$ that appear inside the set $H_{k-i} (\lambda) \cap V_{i+1}(\lambda)$ (resp. inside the set $H_{k-i} (\lambda) \cap V_{i}(\lambda)$).
Recall that $\lambda$ is irreducible if and only if $(n_i,m_i) \in \{0,1,\hdots,k-1-i\}^2$ for all $i \in [k-2]$. Consider $i_0 \in [k-2]$. The number $n_{i_0}$ is precisely the number of saturated columns of height $i_0+1$ of the partial $k$-shape $s^1(f) = \partial^k(\lambda)$. Since $s^1(f)$ is saturated by construction, this number is the quantity $z_{2i_0}(f) < k-i_0$ according to Definition \ref{definitionvarphi}. This statement being true for any $i_0 \in [k-2]$, in particular, if $i_0 > 1$, there are $n_{i_0-1} = z_{2i_0-2}(f)$ columns of height $i_0$ and label $1$ in $s^1(f)$, thence the quantity $m_{i_0}$ is precisely the number $z_{2i_0-1}(f) < k-i_0$ of columns of height $i_0$ and label $2$. Also, the columns of height $1$ are necessarily labeled by $2$, so $m_1 = z_1(f) < k-1$. Consequently, the $k$-shape $\lambda$ is irreducible. Finally, for all $i \in [k-2]$, we have the equivalence $f(2i) = 2i \Leftrightarrow z_{2i}(f) = 0$. Indeed, if $f(2i) = 2i$ then by definition $z_{2i}(f) =  f(2i)/2 - i = 0$. Reciprocally, if $f(2i) > 2i$, then either $z_{2i}(f)$ is defined in the context $(1)$ of Definition \ref{definitionvarphi}, in which case $z_{2i}(f) > 0$, or $z_{2i}(f) = f(2i)/2 - i > 0$. Therefore, the equivalence is true and exactly translates into $\overrightarrow{fix}(f) = \overrightarrow{fr}(\lambda)$.
\end{proof}

\subsection{Algorithm $\phi : IS_k \rightarrow SP_{k-1}$}
\label{sec:phi}

\begin{defi}
Let $\lambda$ be an irreducible $k$-shape. For all $i \in [k-2]$, we denote by $x_{i}(\lambda)$ the number of horizontal steps of the $k$-rim of $\lambda$ inside the set $H_{k-i}(\lambda) \cap V_{i+1}(\lambda)$, and by $y_i(\lambda)$ the number of horizontal steps inside the set $V_i(\lambda) \backslash H_{k+1-i}(\lambda) \cap V_{i}(\lambda) = \bigsqcup_{j=1}^{k-i} H_j(\lambda) \cap V_i(\lambda)$. Finally, for all $j \in [2k-4]$, we set 
$$z_j(\lambda) = \begin{cases} y_i(\lambda) & \text{ if $j=2i-1$},\\
x_{i}(\lambda) & \text{ if $j = 2i$}. \end{cases}$$
\end{defi}
\hspace*{-5.8mm}
For example, if $\lambda$ is the irreducible $6$-shape represented in Figure \ref{exempleirreducible}, then $(z_j(\lambda))_{j \in [8]}=(3,2,1,3,2,0,0,1)$.
Note that in general, if $\lambda$ is an irreducible $k$-shape and $(t_1,t_2,\hdots,t_{k-2}) = \overrightarrow{fr}(\lambda)$, then $t_i = 1$ if and only if $x_{i}(\lambda) = 0$, for all $i$.

\begin{lem} \label{kshapeirreductiblecondition}
For all $\lambda \in IS_k$ and for all $j \in [2k-4]$, we have
$$z_j(\lambda) \in \{0,1,\hdots,k-1-\lceil j/2 \rceil\}.$$
\end{lem}

\begin{proof}
By definition of an irreducible $k$-shape, we automatically have $z_{2i}(\lambda) = x_{i}(\lambda) < k-i$ for all $i \in [k-2]$. The proof of $z_{2i-1}(\lambda) = y_i(\lambda) < k-i$ is less straightforward. Suppose that $y_i(\lambda) \geq k-i$. Let $C_0$ be the first column (from left to right) of $\bigsqcup_{j=1}^{k-i} H_j(\lambda) \cap V_i(\lambda)$, let $R_0$ the row in which $C_0$ is rooted, and let $R_1$ be the row beneath $R_1$. We denote by $l \in [y_i(\lambda)]$ the number of consecutive columns of $\bigsqcup_{j=1}^{k-i} H_j(\lambda) \cap V_i(\lambda)$ whose bottom cells are located $R_0$, and $l'$ the length of $R_1$ (see Figure \ref{preuveirreduciblelessthan1}).
\begin{figure}[!h]
\centering
\begin{minipage}{.5\textwidth}
  \centering
\includegraphics[width=8cm]{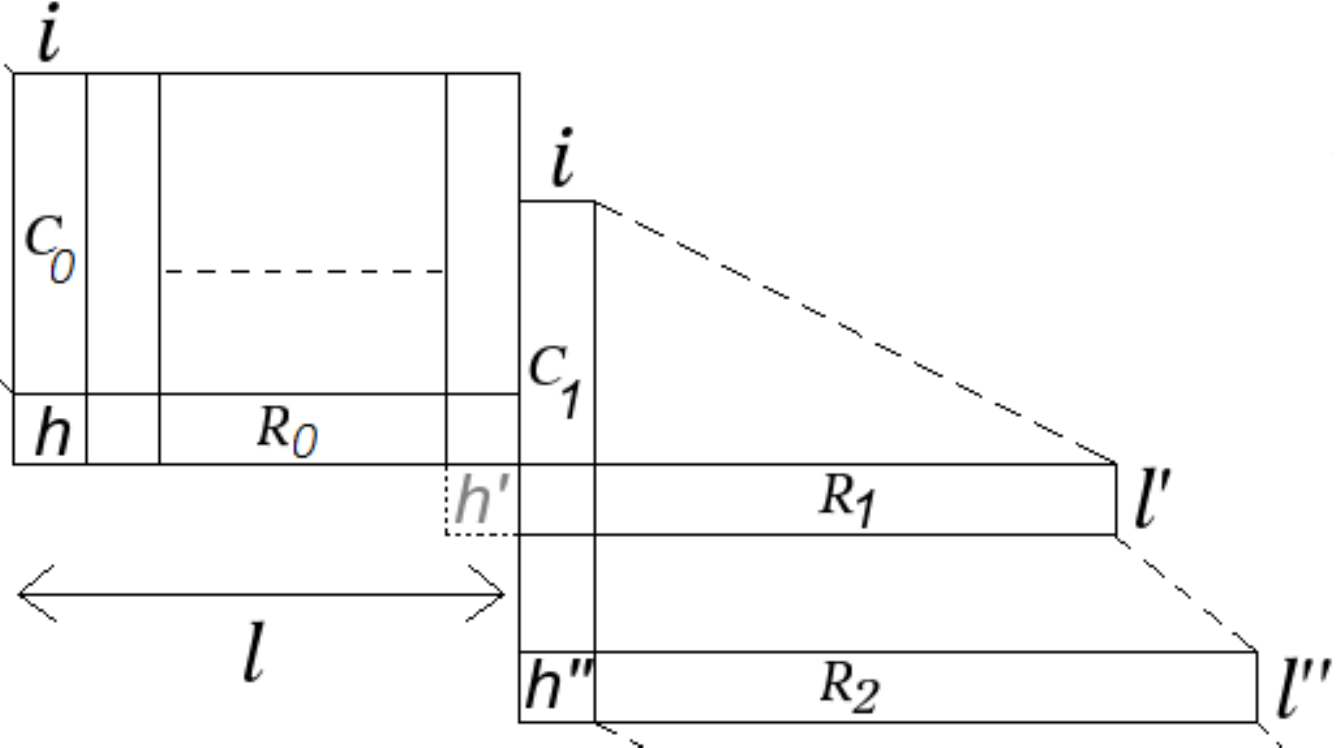}
\caption{}
\label{preuveirreduciblelessthan1}
\end{minipage}%
\begin{minipage}{.5\textwidth}
  \centering
\includegraphics[width=7.5cm]{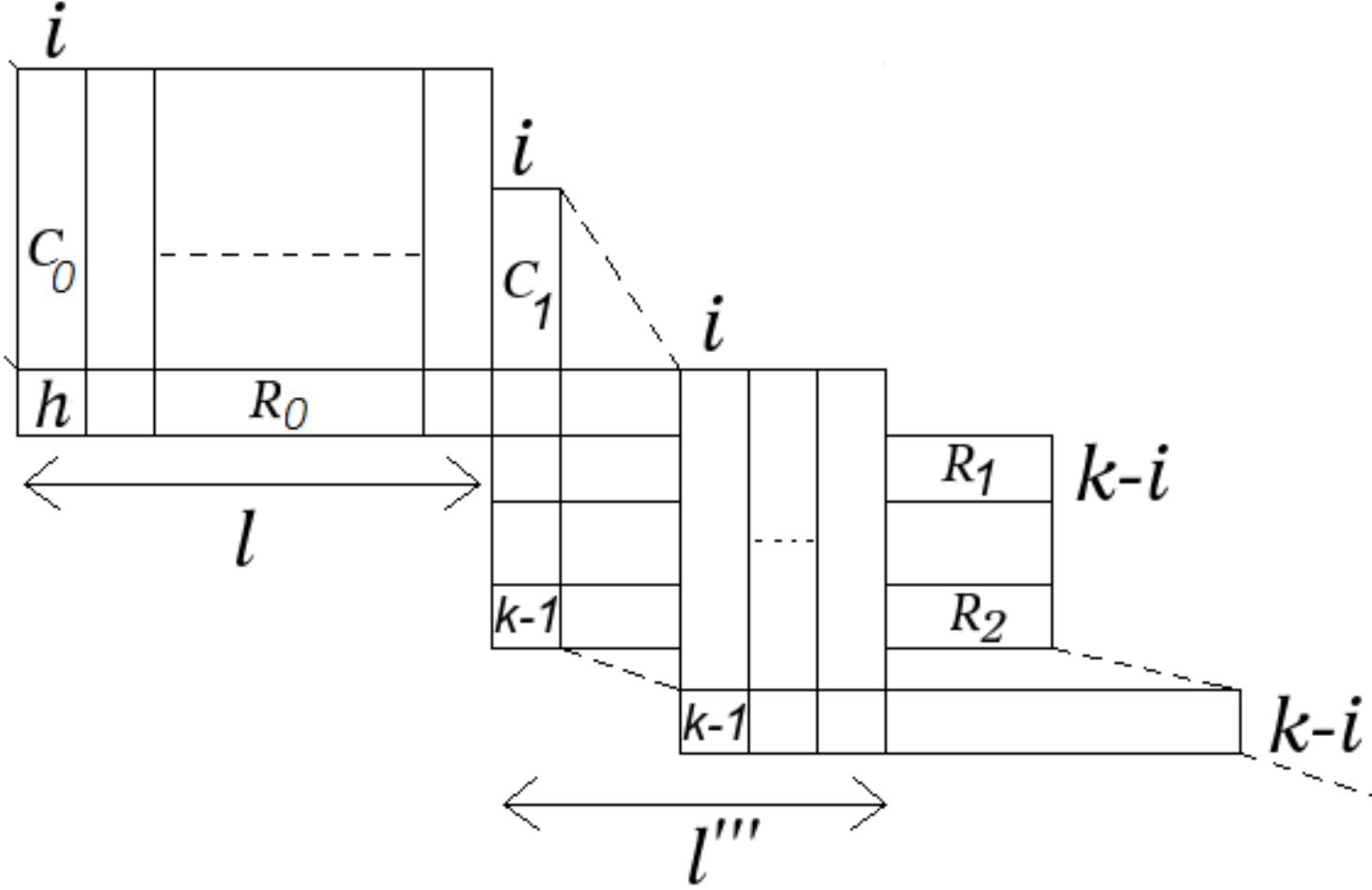}
\caption{}
\label{preuveirreduciblelessthan2}
\end{minipage}
\centering
\begin{minipage}{0.7\textwidth}
\includegraphics[width=12cm]{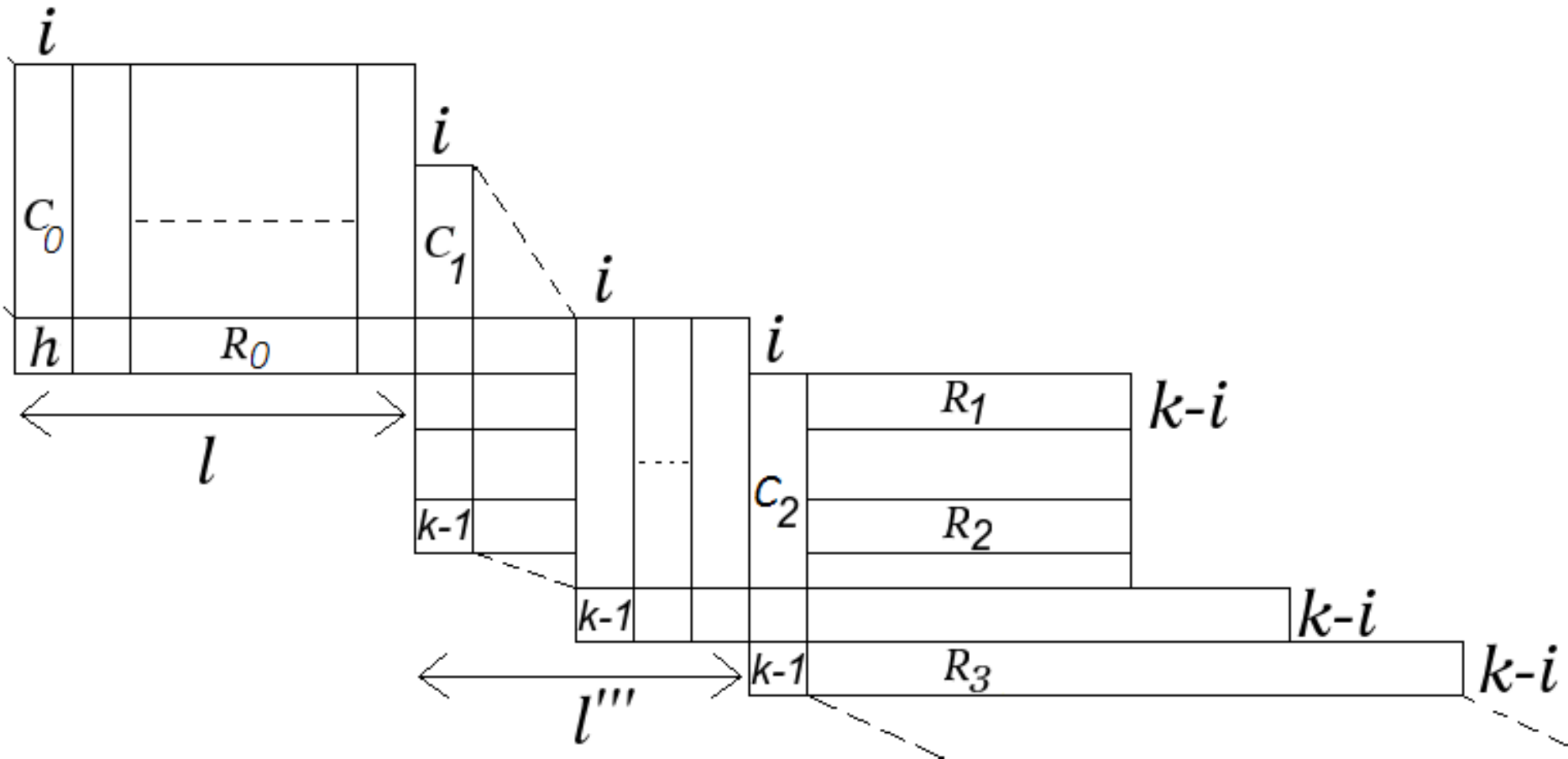}
\caption{}
\label{preuveirreduciblelessthan3}
\end{minipage}
\end{figure}
\\
The hook cell $h' = i+l'+1$ of the cell glued to the left of $R_1$ exceeds $k$, thence $l' \geq k-i$.
Now suppose that $l \geq k-i$: then, the hook length $h \leq k-1$ of the bottom cell $c_0$ of $C_0$ is such that $h \geq i+l-1 \geq k-1$ hence $h = k-1$. As a result, the first $l$ columns of $\bigsqcup_{j=1}^{k-i} H_j(\lambda) \cap V_i(\lambda)$ are in fact located in $H_{k-i}(\lambda) \cap V_i(\lambda)$, therefore $|H_{k-i}(\lambda) \cap V_i(\lambda)| \geq l \geq k-i$, which contradicts the irreducibility of $\lambda$. So $l \leq k-i-1 < y_i(\lambda)$. As a consequence, there exists a column of $\bigsqcup_{j=1}^{k-i} H_j(\lambda) \cap V_i(\lambda)$ which intersects $R_1$; in particular, consider the first column $C_1$ that does so, then its bottom cell $c_1$ is located in a row $R_2$ (whose length is denoted by $l''\geq l'$) and $c_1$ is hooked lengthed by $h'' = i+l''-1 < k $, thence $k-i \leq l' \leq l'' \leq k-i$, which implies $l' = l'' = k-i$. Now, let $l''' \in \{0,1,\hdots,k-i\}$ be the number of columns of $\bigsqcup_{j=1}^{k-i} H_j(\lambda) \cap V_i(\lambda)$ intersecting $R_1$ but whose top cells are not located in $R_1$ (see Figure \ref{preuveirreduciblelessthan2}).
Since $h = i-1+l+l''' \leq k-1$, we obtain $l + l''' \leq k-i$. With precision, we must have $l+l''' \leq k-i-1$: otherwise, we would have $l+l''' = k-i$ and $h = k-1$ which implies the first $l+l''' = k-i$ columns of $\bigsqcup_{j=1}^{k-i} H_j(\lambda) \cap V_i(\lambda)$ are located in $H_{k-i}(\lambda) \cap V_i(\lambda)$, which cannot be because $\lambda$ is irreducible. So $l+l''' < k-i \leq y_i(\lambda)$. It means there exists a column $C_2$ of $\bigsqcup_{j=1}^{k-i} H_j(\lambda) \cap V_i(\lambda)$ whose top box is located in $R_1$, which forces its bottom cell to be located in a row $R_3$ of length $k-i$ (see Figure \ref{preuveirreduciblelessthan3}) because $rs(\lambda)$ is a partition. But then the bottom cell of every column intersecting $R_1$ is located in a row of length $k-i$, therefore the bottom cells of those $k-i$ columns are elements of the set $H_{k-i}(\lambda) \cap V_i(\lambda)$, which cannot be because $\lambda$ is irreducible and the length of $R_1$ is $k-i$. As a conclusion, it is necessary that $y_i(\lambda) < k-i$.
\end{proof}

\begin{defi}
Let $\lambda \in IS_k$. We define a sequence $(s^j(\lambda))_{j \in [2k-3]}$ of partial $k$-shapes by $s^{2k-3}(\lambda) = \varnothing$ and 
$$s^j(\lambda) = s^{j+1}(\lambda) \oplus_{t(j)}^k \lceil (j+1)/2 \rceil^{z_j(\lambda)}.$$
\end{defi}

\begin{lem} \label{slambdakboundary}
We have $s^1(\lambda) = \partial^k(\lambda)$ for all $\lambda \in IS_k$.
\end{lem}

\begin{proof}
Let $n$ be the number of columns of $[\lambda]$, which is obviously the same as for $\partial^k(\lambda)$ and $s^1(\lambda)$. For all $q \in [n]$, we define $\partial^k(\lambda)_q$ (respectively $s^1(\lambda)_q$) as the skew partition (resp. labeled skew partition) obtained by considering the $q$ first columns (from right to left) of $\partial^k(\lambda)$ (resp. $s^1(\lambda)$). The idea is to prove that $\partial^k(\lambda)_q = s^1(\lambda)_q$ for all $q \in [n]$ by induction (the statement being obvious for $q = 1$). In particular, for $q = n$, we will obtain $\partial^k(\lambda) = s^1(\lambda)$. Suppose that $\partial^k(\lambda)_q = s^1(\lambda)_q$ for some $q \geq 1$. The $(q+1)$-th column $C_{q+1}$ (whose bottom cell is denoted by $c_{q+1}$) of $\partial^k(\lambda)_{q+1}$ (from right to left) is glued to the left of $\partial^k(\lambda)_q$, at the unique level such that the hook length $h$ of $c_{q+1}$ doesn't exceed $k$, and the hook length $x$ of the cell beneath $c_{q+1}$ exceeds $k$. Since the hook length of every cell of $s^1(\lambda)$ doesn't exceed $k$, the $(q+1)$-th column $C'_{q+1}$ (whose bottom cell is denoted by $c'_{q+1}$) of $s^1(\lambda)_{q+1}$ is necessarily positionned above or at the same level as $C_{q+1}$ (see Figure \ref{preuveomegagrid2} where $C_{q+1}$ [resp. $C'_{q+1}$] has been drawed in black [resp. in red] and whose bottom cell is labeled by its hook length $h$ [resp. $h'$]).
\begin{figure}[!h]
\centering
\includegraphics[width=1.5cm]{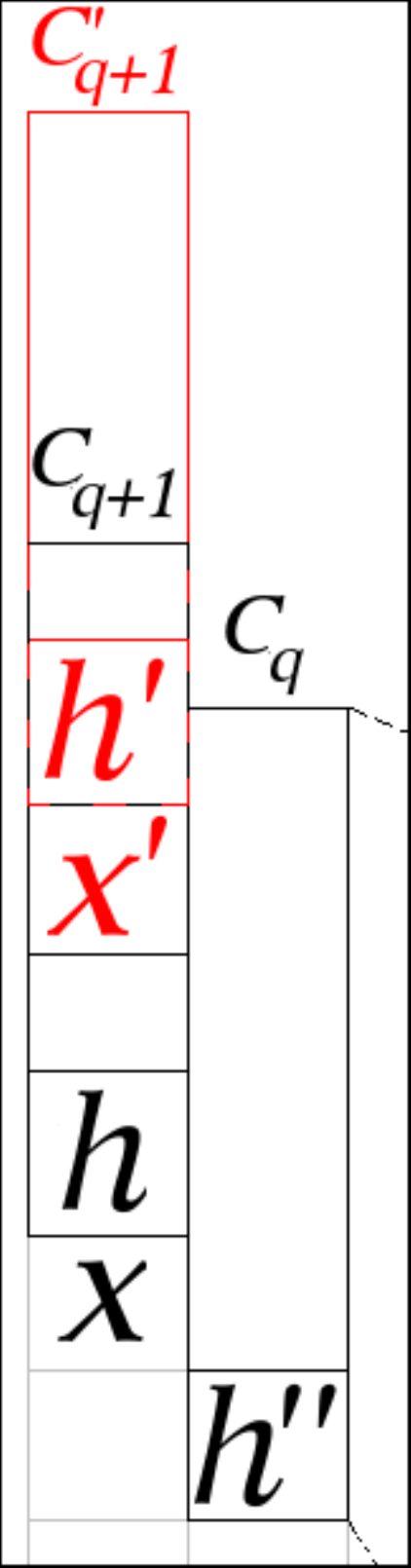}
  \captionof{figure}{Skew partitions $\partial^k(\lambda)_{q+1}$ and $s^1(\lambda)_{q+1}$.}
\label{preuveomegagrid2}
\end{figure}
\\
Now, suppose that the columns $C_{q+1}$ and $C'_{q+1}$ are not at the same level. In particular, the hook length $x'$ of the cell beneath $c'_{q+1}$ is such that $x' \leq h \leq k$. Also, the bottom cell $c'_q$ of the $q$-th column $C'_q$ of $s^1(\lambda)_{q+1}$ is a corner.
\begin{enumerate}
\item If $C'_{q+1}$ has been lifted in the context $(1)$ of Definition \ref{sumskew}, then since $x' \leq k$, the label of $C'_{q+1}$ is necessarily $2$, and $x' = k$. Consequently, the cell $c_{q+1}$ is in fact the cell labeled by $x'$, and $h = x' = k$. In particular, this implies $C'_{q+1}$ is labeled by $1$, which is absurd because $C_{q+1}$ and $C'_{q+1}$ must have the same label.
\item If $C'_{q+1}$ has been lifted in the context $(2)$ of Definition \ref{sumskew}, then in particular $C'_q$ is a column labeled by $1$. Consequently, the hook length of $c'_q$, which is the same as $c_q$ because $s^1(\lambda)_q = \partial^k(\lambda)_q$, is the integer $h'' = k$. Since $rs^k(\lambda)$ and $cs^k(\lambda)$ are partitions, this implies $h > h'' = k$, which is absurd.
\item Therefore, the column $C'_{q+1}$ has necessarily been lifted in the context $(3)$ of Definition \ref{sumskew}. According to Remark \ref{explicitation3}, it implies :
\begin{enumerate}
\item $C'_{q+1}$ and $C'_q$ have the same height and the same label $2$ (and $h'' = k-1$);
\item $C'_{q+1}$ is located one cell higher than $C'_q$. \end{enumerate}
In particular, from $(b)$, since $C'_{q+1}$ is supposed to be located at a higher level than $C_{q+1}$, then the cell $c_{q+1}$ is glued to the left of the cell $c_q$. Since $h'' = k-1$, we obtain $h = k$, which is in contradiction with $C'_{q+1}$ being labeled by $2$.
\end{enumerate}
So $C_{q+1}$ and $C'_{q+1}$ are located at the same level, thence $\partial^k(\lambda)_{q+1} = s^1(\lambda)_{q+1}$.
\end{proof}
\\ \ \\
Notice that Lemma \ref{slambdakboundary} is obvious if we know that $\lambda = \varphi(f)$ for some surjective pistol $f \in SP_{k-1}$, because in that case $s^j(\lambda) = s^j(f)$ for all $j$. 

\begin{defi}[Algorithm $\phi$] \label{definitionphi}
Let $\lambda \in IS_k$. We define $m(\lambda) \in \{0,1,\hdots,k-2\}$ and $$1\leq i_1(\lambda) < i_2(\lambda) < \hdots i_m(\lambda) \leq k-2$$ such that $$\{i_1(\lambda),i_2(\lambda),\hdots,i_{m(\lambda)}(\lambda)\} = \{i\in[k-2],x_{i} (\lambda) > 0\}$$ (this set may be empty). For all $p \in [m(\lambda)]$, let
$$j_p(\lambda) = \max \{j \in [2i_p(\lambda)-1],\text{$s^j(\lambda)$ is saturated in $i_p(\lambda)$}\}.$$
Let $L(\lambda)= [2k-4]$. For $j$ from $1$ to $2k-4$, if $j = j_p(\lambda)$ for some $p \in [m(\lambda)]$, and if there is no $j' \in L(\lambda)$ such that $j' < j$ and $\lceil j'/2 \rceil + z_{j'} = i_p(\lambda)$, then we set $L(\lambda) := L(\lambda) \backslash \{j_p(\lambda)\}$. Now we define $\phi(\lambda) \in \N^{[2k-2]}$ as the following: the integers $\phi(\lambda)(2k-2)$ and $\phi(\lambda)(2k-3)$ are defined as $2k-2$; afterwards, let $j \in [2k-4]$. 
\begin{itemize}
\item If $j \in L(\lambda)$ then $\phi(\lambda)(j)$ is defined as $2(\lceil j/2 \rceil + z_j(\lambda))$.
\item Else there exists a unique $p \in [m(\lambda)]$ such that $j = j_p(\lambda)$, and we define $\phi(\lambda)(j)$ as $2 i_p(\lambda)$.
\end{itemize}
\end{defi}

\begin{prop} \label{phipistol}
For all $\lambda \in IS_k$, the map $\phi(\lambda)$ is a surjective pistol of height $k-1$, such that
$\overrightarrow{fix}(\phi(\lambda)) = \overrightarrow{fr}(\lambda)$.
\end{prop}

For example, consider the irreducible $6$-shape $\lambda$ of Figure \ref{exempleirreducible}, such that  $(z_j(\lambda))_{j \in [8]} = (3,2,1,3,2,0,0,1)$. In particular $(x_1(\lambda),x_2(\lambda),x_3(\lambda),x_4(\lambda)) = (2,3,0,1)$ so $m(\lambda) = 3$ and $(i_1(\lambda),i_2(\lambda),i_3(\lambda)) = (1,2,4)$. Moreover, by considering the sequence of partial $6$-shapes $(s^8(\lambda),\hdots,s^1(\lambda))$, which is in fact the sequence $(s^8(f),\hdots,s^1(f))$ depicted in Figure \ref{tableau} (with $f = (2,8,4,10,10,6,8,10,10,10) \in SP_5$) because $\lambda = \varphi(f)$, we obtain $(j_2(\lambda),j_3(\lambda),j_1(\lambda)) = (3,2,1)$. Applying the algorithm of Definition \ref{definitionphi} on $L(\lambda) = [8]$, we quickly obtain $L(\lambda) = \{4,5,6,7,8\}$. Consequently, if $g = \phi(\lambda)$, then automatically $g(10)=g(9)=10$, afterwards $(g(1),g(2),g(3)) = (g(j_1(\lambda)),g(j_3(\lambda),g(j_2(\lambda)) = (2i_1(\lambda),2i_3(\lambda),2i_2(\lambda)) = (2,8,4)$ since $j_p(\lambda) \not\in L(\lambda)$ for all $p \in [3]$, and $g(j) = 2(\lceil j/2 \rceil + z_j(\lambda))$ for all $j \in L(\lambda)$. Finally, we obtain $g = (2,8,4,10,10,6,8,10,10,10) = f$ (and $\overrightarrow{fix}(g) = \overrightarrow{fr}(\lambda)$).
\\ \ \\
\textbf{Proof of Proposition \ref{phipistol}.}
Let $\lambda \in IS_k$ and $f = \phi(\lambda)$. We know that $f(2k-2) = f(2k-3) = 2k-4$. Consider $j \in [2k-4]$.
\begin{enumerate}
\item If $j = j_p(\lambda)$ for some $p \in [m(\lambda)]$ and if $j \not\in L(\lambda)$, then $f(j) = 2i_p(\lambda)$. By definition $2i_p(\lambda) > j_p(\lambda)$, so $2k-2 \geq f(j) > j$.
\item Else $f(j) = 2(\lceil j/2 \rceil + z_j(\lambda))$, so $2k-2 \geq f(j) \geq j$ following Lemma \ref{kshapeirreductiblecondition}.
\end{enumerate} 
Consequently $f$ is a map $[2k-2] \rightarrow \{2,4,\hdots,2k-2\}$ such that $f(j) \geq j$ for all $j \in [2k-2]$.
Now, we prove that $f$ is surjective. We know that $2k-2 = f(2k-2)$. Let $i \in [k-2]$.
\begin{itemize}
\item If $i = i_p(\lambda)$ for some $p \in [m(\lambda)]$, then either $j_p(\lambda) \not\in L(\lambda)$, in which case $2i = f(j_p(\lambda))$, or there exists $j < j_p(\lambda)$ in $L(\lambda)$ such that $\lceil j/2 \rceil + z_j = i$, in which case $2i = f(j)$.
\item Else $z_{2i}(\lambda) = 0$, which implies that $2i$ cannot be equal to any $j_p(\lambda)$ because $s^{2i}(\lambda) = s^{2i+1}(\lambda) \oplus_1^k (i+1)^{z_j(\lambda)} = s^{2i+1}(\lambda)$. Consequently $2i \in L(\lambda)$, thence  $f(2i) = 2(i+z_{2i}(\lambda)) = 2i$.
\end{itemize} 
Therefore $f \in SP_{k-1}$. Finally, for all $i \in [k-2]$, we have just proved that $z_{2i}(\lambda) = 0$ implies $f(2i) = 2i$. Reciprocally, if $f(2i) = 2i$, then necessarily $2i \in L(\lambda)$ (otherwise $2i$ would be $j_p(\lambda)$ for some $p$ and $f(2i)$ would be $2i_p(\lambda) > j_p(\lambda) = 2i$), meaning $2i = f(2i) = 2(i+z_{2i}(\lambda))$ thence $z_{2i}(\lambda) = 0$. The equivalence $z_{2i}(\lambda) = 0 \Leftrightarrow f(2i) = 2i$ for all $i \in [k-2]$ exactly translates into $\overrightarrow{fr}(\lambda) = \overrightarrow{fix}(f)$.\qed

\subsection{Proof of Theorem \ref{theoremebijection}}

At this stage, we know that $\varphi$ is a map $SP_{k-1} \rightarrow IS_k$ that transforms the statistic $\overrightarrow{fix}$ into the statistic $\overrightarrow{fr}$. The bijectivity of $\varphi$ is a consequence of the following proposition.

\begin{prop} \label{bijection}
The maps $\varphi : SP_{k-1} \rightarrow IS_k$ and $\phi : IS_k \rightarrow SP_{k-1}$ are inverse maps.
\end{prop}

\begin{lem} \label{equivalence}
Let $(f,\lambda) \in SP_{k-1} \times IS_k$ such that $\lambda = \varphi(f)$ or $f = \phi(\lambda)$. Let $p \in [m(\lambda)]$ and $j^p(\lambda) := \min\{j \in [2k-4],f(j) = 2i_p(\lambda)\}$. The two following assertions are equivalent.
\begin{enumerate}
\item $j_p(\lambda) \not\in L(\lambda)$.
\item $j_p(\lambda) = j^p(\lambda)$.
\end{enumerate}
\end{lem}

\textbf{Proof.}
Let $f \in SP_{k-1}$ and $\lambda = \varphi(f)$. In particular, we have $s^j(\lambda) = s^j(f)$ and $z_j(\lambda) = z_j(f)$ for all $j \in [2k-4]$. For all $p \in [m(\lambda)]$, by Definition \ref{definitionvarphi} the partial $k$-shape $s^{j^p(\lambda)}(f) = s^{j^p(\lambda)}(\lambda)$ is necessarily saturated in $i_p(\lambda)$, thence $j_p(\lambda) \geq j^p(\lambda)$. 
\begin{enumerate}
 \item
 If $j_p(\lambda) \not \in L(\lambda)$, suppose that $j_p(\lambda) > j^p(\lambda)$. Then, the partial $k$-shape $s^{j^p(\lambda)+1}(f) = s^{j^p(\lambda)+1}(\lambda)$ is saturated in $i_p(\lambda)$, meaning the integer $z_{j^p(\lambda)}(f) = z_{j^p(\lambda)}(\lambda)$ is defined as $f(j^p(\lambda))/2 - \lceil j^p(\lambda)/2 \rceil = i_p(\lambda) - \lceil j^p(\lambda)/2 \rceil$. Consequently, since $j_p(\lambda) \not \in L(\lambda)$ and $j^p(\lambda) < j_p(\lambda)$, the integer $j^p(\lambda)$ cannot belong to $L(\lambda)$ either. So $j^p(\lambda) = j_{p_1}(\lambda)$ for some $p_1 \neq p$ because $j_p(\lambda) \neq j^p(\lambda)$. Also, since $f(j_{p_1}(\lambda)) = 2i_p(\lambda) \neq 2i_{p_1(\lambda)}$, then $j_{p_1}(\lambda) > j^{p_1}(\lambda)$ (and $j_{p_1}(\lambda) = j^p(\lambda) \not\in L(\lambda)$). By iterating, we build an infinite decreasing sequence $(j^{p_i}(\lambda))_{i \geq 1}$ of distinct elements of $[2k-4]$, which is absurd. Therefore, it is necessary that $j_p(\lambda) = j^p(\lambda)$.
 \item Reciprocally, if $j_p(\lambda) = j^p(\lambda)$, suppose that $j_p(\lambda) \in L(\lambda)$. Then, there exists $j \in L(\lambda)$ such that $j < j_p(\lambda)$ and $z_j(\lambda) = i_p(\lambda) - \lceil j/2 \rceil$. Let $i \in [k-1]$ such that $f(j) = 2i$ (since $j < j_p(\lambda) = j^p(\lambda)$, we know that $i \neq i_p(\lambda)$). Suppose $s^j(f)$ is defined in the context $(1)$ of Definition \ref{definitionvarphi}. In particular $i = i_{p_1}(\lambda)$ for some $p_1 \in [m(\lambda)]$ (because $f(2i) > 2i$ which implies $z_{2i}(\lambda) = z_{2i}(f) > 0$), and $j = j^{p_1}(\lambda)$ and $s^{j+1}(f)$ is not saturated in $i$. Then, by definition the partial $k$-shape $s^j(f)$ is the first partial $k$-shape to be saturated in $i_{p_1}(\lambda)$ in the sequence $(s^{2k-4}(f) = s^{2k-4}(\lambda),\hdots,s^1(f) = s^1(\lambda))$, meaning $j = j_{p_1}(\lambda)$. To sum up, the integer $j = j^{p_1}(\lambda) = j_{p_1}(\lambda)$ doesn't belong to $L(\lambda)$, and $p_1 \neq p$ because $f(j) = 2i_{p_1}(\lambda)$ and $j < j^p(\lambda) = j_p(\lambda)$. By iterating, we build an infinite decreasing sequence $(j^{p_i}(\lambda))_{i \geq 1}$ of elements of $[2k-4]$, which is absurd. So $s^{j+1}(f)$ is necessarily defined in the context $(2)$ of Definition \ref{definitionvarphi}, meaning $z_j(f) = f(j)/2 - \lceil j/2 \rceil)$. Since $z_j(f) = z_j(\lambda) = i_p(\lambda) - \lceil j/2 \rceil$, we obtain $f(j) = 2i_p(\lambda)$, which is in contradiction with $j_p(\lambda) = j^p(\lambda) > j$. As a conclusion, it is necessary that $j_p(\lambda) \not \in L(\lambda)$.
 \end{enumerate}
Now let $\lambda \in IS_k$ and $f = \phi(\lambda)$. We consider $p \in [m(\lambda)]$.
\begin{enumerate}
\item If $j_p(\lambda) \not \in L(\lambda)$, suppose that $j_p(\lambda) \neq j^p(\lambda)$. Then, by definition $f(j_p(\lambda)) = 2i_p(\lambda)$, meaning $j_p(\lambda) > j^p(\lambda)$. Suppose now that $j^p(\lambda) \in L(\lambda)$, then $2 i_p(\lambda) = f(j^p(\lambda)) = 2(\lceil j^p(\lambda)/2 \rceil + z_{j^p(\lambda)}(\lambda))$. As a result, we obtain $z_{j^p(\lambda)}(\lambda) = i_p(\lambda) - \lceil j^p(\lambda)/2 \rceil$, which is in contradiction with $j_p(\lambda) \not \in L(\lambda)$. So $j^p(\lambda) \not \in L(\lambda)$, which implies $j^p(\lambda) = j_{p_1}(\lambda)$ for some $p_1 \neq p$, and necessarily $j_{p_1}(\lambda) \neq j^{p_1}(\lambda)$ since $f(j_{p_1}(\lambda)) = 2 i_p(\lambda) \neq 2i_{p_1}(\lambda)$. By iterating, we build a sequence $(j^{p_i}(\lambda))_{i \geq 1}$ of distinct elements of $[2k-4]$, which is absurd. So $j_p(\lambda) = j^p(\lambda)$.
\item Reciprocally, if $j_p(\lambda) = j^p(\lambda)$, suppose that $j_p(\lambda) \in L(\lambda)$. Then, there exists $j \in L(\lambda)$ such that $j < j_p(\lambda)$ and $z_j(\lambda) = i_p(\lambda) - \lceil j/2 \rceil$. Let $i \in [2k-4]$ such that $f(j) = 2i$. Because $j^p(\lambda) > j$, we have $i \neq i_p(\lambda)$. And since $j \in L(\lambda)$, we obtain $2i = f(j) = 2(\lceil j/2 \rceil + z_j(\lambda)) = 2i_p(\lambda)$, which is absurd. So $j_p(\lambda) \not \in L(\lambda)$.\qed
\\
\end{enumerate}
\textbf{Proof of Proposition \ref{bijection}.} Let $f \in SP_{k-1}$ and $\lambda = \varphi(f)$ and $g = \phi(\lambda)$. Let $j \in [2k-4]$ and $i \in [k-1]$ such that $f(j) = 2i$. 
\begin{enumerate}
\item If $s^j(f)$ is defined in the context $(1)$ of Definition \ref{definitionvarphi}, then there exists $p \in [m(\lambda)]$ such that $i = i_p(\lambda)$ and $j = j^p(\lambda) = j_p(\lambda)$. Consequently, in view of Lemma \ref{equivalence} with $\lambda = \varphi(f)$, we know that $j \not\in L(\lambda)$, implying $g(j) = g(j_p(\lambda)) = 2i_p(\lambda) = 2i = f(j)$. 
\item If $s^j(f)$ is defined in the context $(2)$ of Definition \ref{definitionvarphi}, then $z_j(f) = f(j)/2 - \lceil j/2 \rceil = i - \lceil j/2 \rceil$. Now it is necessary that $j \in L(\lambda)$: otherwise $j = j_p(\lambda)$ for some $p \in [m(\lambda)]$, and from Lemma \ref{equivalence} we would have $j = j_p(\lambda) = j^p(\lambda)$, which is impossible because we are in the context $(2)$ of Definition \ref{definitionvarphi}. So $j \in L(\lambda)$, implying $g(j) = 2(\lceil j/2 \rceil + z_j(\lambda)) = 2(\lceil j/2 \rceil + z_j(f)) = 2i = f(j)$. 
\end{enumerate}
As a conclusion, we obtain $g = f$ so $\phi \circ \varphi$ is the identity map of $SP_{k-1}$.\\
Reciprocally, let $\mu \in IS_k$ and $h = \phi(\mu)$. We are going to prove by induction that $s^j(\mu) = s^j(h)$ for all $j \in [2k-3]$. By definition $s^{2k-3}(\mu) = s^{2k-3}(h) = \varnothing$. Suppose that $s^{j+1}(\mu) = s^{j+1}(h)$ for some $j \in [2k-4]$. 
\begin{enumerate}
\item If $s^j(h)$ is defined in the context $(1)$ of Definition \ref{varphikshape}, then there exists $p \in [m(\lambda)]$ such that $h(j) = 2i_p(\mu)$, such that $j= j^p(\mu)$ and such that $s^{j+1}(h)$ is not saturated in $i_p(\mu)$. Since the partial $k$-shape $s^{j+1}(\mu) = s^{j+1}(h)$ is not saturated in $i_p(\mu)$, by definition $j \geq j_p(\mu)$. Suppose that $j > j_p(\mu)$. Since $j = j^p(\mu)$, we know from Lemma \ref{equivalence} (with $\lambda = \mu$ and $f = \phi(\lambda) = h$) that $j_p(\mu) \in L(\mu)$. It means there exists $j' < j_p(\mu) < j$ such that $j' \in L(\mu)$ and $\lceil j'/2 \rceil + z_{j'}(\mu) = i_p(\mu)$, implying $h(j') = 2i_p(\mu) = h(j)$, wich contradicts $j = j^p(\mu)$. So $j = j_p(\mu)$, therefore $s^j(\mu)$ is saturated in $i_p(\mu)$. But since we are in the context $(1)$ of Definition \ref{definitionvarphi}, the partial $k$-shape $s^j(h)$ is defined as $s^{j+1}(h) \oplus_{t(j)}^k (\lceil (j+1)/2 \rceil^{z_j(h)}$ where $z_j(h)$ is the unique integer $z \in [k-1 - \lceil j/2 \rceil]$ such that $s^{j+1}(h) \oplus_{t(j)}^k (\lceil (j+1)/2 \rceil^{z}$ is saturated in $i_p(\mu)$. Since the partial $k$-shape $s^j(\mu) = s^{j+1}(\mu) \oplus_{t(j)}^k (\lceil (j+1)/2 \rceil^{z_j(\mu)} = s^{j+1}(h) \oplus_{t(j)}^k (\lceil (j+1)/2 \rceil^{z_j(\mu)}$ is saturated in $i_p(\mu)$, we obtain $z_j(\mu) = z_j(h)$ and $s^j(\mu) = s^j(h)$. 
\item If $s^j(h)$ is defined in the context $(2)$ of Definition \ref{definitionvarphi}, then $s^j(h) = s^{j+1}(\mu) \oplus_{t(j)}^k\lceil (j+1)/2 \rceil^{z_j(h)}$ with $z_j(h) = h(j)/2 - \lceil j/2 \rceil$. Now either $h(j) = 2(\lceil j/2 \rceil + z_j(\mu))$, in which case we obtain $z_j(h) = z_j(\mu)$, or $h(j) = 2i_p(\mu)$ for some $p \in [m(\mu)]$ such that $j = j_p(\mu) \not \in L(\mu))$. In view of Lemma \ref{equivalence}, it means $j = j^p(\mu)$, which cannot happen because otherwise we would be in the context $(1)$ of Definition \ref{definitionvarphi}. So $z_j(h) = z_j(\mu)$ and $s^j(h) =s^j(\mu)$.
\end{enumerate}
By induction, we obtain $s^1(\mu) = s^1(h)$, thence $\mu = \varphi(h)$. Consequently, the map $\varphi \circ \phi$ is the identity map of $IS_k$.\qed

\section{Extensions}
\label{sec:workinprogress}

Dumont and Foata~\cite{DF} introduced a refinement of Gandhi polynomials $(Q_{2k}(x))_{k \geq 1}$ through the polynomial sequence $(F_k(x,y,z))_{k \geq 1}$ defined by $F_1(x,y,z) = 1$ and
$$F_{k+1}(x,y,z) = (x+y)(x+z)F_k(x+1,y,z) - x^2 F_k(x,y,z).$$
Note that $x^2F_k(x,1,1) = Q_{2k}(x)$ for all $k \geq 1$ in view of Formula \ref{inductionformulagandhi}. Now, for all $k \geq 2$ and  $f \in SP_{k}$, let $max(f)$ be the number of \textit{maximal} points of $f$ (points $j \in [2k-2]$ such that $f(j) = 2k$) and $pro(f)$ the number of \textit{prominent} points (points $j \in [2k-2]$ such that $f(i) < f(j)$ for all $i \in [j-1]$). For example, if $f$ is the surjective pistol $(2,4,4,8,8,6,8,8) \in SP_4$ depicted in Figure \ref{exempletableau}, then the maximal points of $f$ are $\{4,5\}$, and its prominent points are $\{2,4\}$. Dumont and Foata gave a combinatorial interpretation of $F_k(x,y,z)$ in terms of surjective pistols.
\begin{theo}[\cite{DF}] \label{theoremeDF}
For all $k \geq 2$, the Dumont-Foata polynomial $F_{k}(x,y,z)$ is symmetrical, and is generated by $SP_k$:
$$F_{k}(x,y,z) = \sum_{f \in SP_{k}} x^{max(f)}y^{fix(f)}z^{pro(f)}.$$
\end{theo}
In 1996, Han~\cite{H} gave another interpretation by introducing the statistic $sur(f)$ defined as the number of \textit{surfixed} points of $f \in SP_{k}$ (points $j \in [2k-2]$ such that $f(j) = j+1$; for example, the surfixed points of the surjective pistol $f \in SP_4$ of Figure \ref{exempletableau} are $\{1,3\}$).
\begin{theo}[\cite{H}] \label{theoremeHan}
For all $k \geq 2$, the Dumont-Foata polynomial $F_k(x,y,z)$ has the following combinatorial interpretation:
$$F_{k}(x,y,z) = \sum_{f \in SP_{k}} x^{max(f)}y^{fix(f)}z^{sur(f)}.$$
\end{theo}
Theorem \ref{theoremegandhi} then appears as a particular case of Theorem \ref{theoremeDF} or Theorem \ref{theoremeHan} by setting $x=z=1$ (and by applying the symmetry of $F_{k}(x,y,z)$).
Furthermore, for all $f \in SP_k$ and $j \in [2k-2]$, we say that $j$ is a  \textit{lined} point of $f$ if there exists $j' \in [2k-2] \backslash \{j\}$ such that $f(j) = f(j')$. We define $mo(f)$ (resp. $me(f)$) as the number of odd (resp. even) maximal points of $f$, and $fl(f)$ (resp. $fnl(f)$) as the number of lined (resp. non lined) fixed points of $f$, and $sl(f)$ (resp. $snl(f)$) as the number of lined (resp. non lined) surfixed points of $f$. Dumont~\cite{Dumont4} defined generalized Dumont-Foata polynomials $(\Gamma_k(x,y,z,\bar{x},\bar{y},\bar{z}))_{k \geq 1}$ by
$$\Gamma_{k}(x,y,z,\bar{x},\bar{y},\bar{z}) = \sum_{f \in SP_{k}} x^{mo(f)} y^{fl(f)} z^{snl(f)} \bar{x}^{me(f)} \bar{y}^{fnl(f)} \bar{z}^{sl(f)}.$$
This a refinement of Dumont-Foata polynomials, considering $\Gamma_k(x,y,z,x,y,z) = F_k(x,y,z)$. Dumont conjectured the following induction formula: $\Gamma_{1}(x,y,z,\bar{x},\bar{y},\bar{z}) = 1$ and
\begin{multline} \label{formulehorrible}
\Gamma_{k+1}(x,y,z,\bar{x},\bar{y},\bar{z}) = (x+\bar{z})(y+\bar{x}) \Gamma_{k}(x+1,y,z,\bar{x}+1,\bar{y},\bar{z}) \\
+ (x(\bar{y}-y) + \bar{x}(z-\bar{z}) - x \bar{x}) \Gamma_{k}(x,y,z,\bar{x},\bar{y},\bar{z}).
\end{multline}
This was proven independently by Randrianarivony~\cite{Ran} and Zeng~\cite{Zeng}. See also \cite{JV} for a new combinatorial interpretation of $\Gamma_{k}(x,y,z,\bar{x},\bar{y},\bar{z})$.\\ \ \\
Now, let $f \in SP_{k-1}$ and $\lambda = \varphi(f) \in IS_k$. For all $j \in [2k-4]$, we say that $j$ is a \textit{chained} $k$-site of $\lambda$ if $j \not\in L(\lambda)$. Else, we say that it is an \textit{unchained} $k$-site. In view of Lemma \ref{equivalence}, an integer $j \in [2k-4]$ is a chained $k$-site if and only if $j = j_p(\lambda) = j^p(\lambda)$ for some $p \in [m(\lambda)]$, in which case $f(j) = 2i_p(\lambda)$ (the integer $j$ is forced to be mapped to $2i_p(\lambda)$, thence the use of the word \textit{chained}). If $j$ is an unchained $k$-site, by definition $f(j) = 2(\lceil j/2 \rceil + z_j(\lambda))$. Consequently, every statistic of Theorems \ref{theoremeDF}, \ref{theoremeHan} and Formula \ref{formulehorrible} has its own equivalent among irreducible $k$-shapes. However, the objects counted by these statistics are not always easily pictured or formalized. We only give the irreducible $k$-shapes version of Theorem \ref{theoremeHan}.\\
Recall that for all $i \in [k-2]$, the integer $2i$ is a fixed point of $f$ if and only if $2i$ is a free $k$-site of $\lambda$, which is also equivalent to $z_{2i}(\lambda) = 0$. We extend the notion of free $k$-site to any $j \in [2k-4]$: the integer $j$ is said to be a free $k$-site if $z_j(\lambda) = 0$. Notice that free $k$-sites of $\lambda$ are necessarily unchained because $z_j(\lambda) = 0$ implies $s^j(\lambda) = s^{j+1}(\lambda)$ thence $j \neq j_p(\lambda)$ for all $p \in [m(\lambda)]$. We denote by $fro(\lambda)$ the quantity of odd free sites of $\lambda$.
We denote by $ful(\lambda)$ the quantity of \textit{full} $k$-site of $\lambda$ (namely, unchained $k$-sites $j \in L(\lambda)$ such that $z_j(\lambda) = k - 1 - \lceil j/2 \rceil$), and by $sch(\lambda)$ the quantity of \textit{surchained} $k$-sites (chained $k$-sites $j \in [2k-4]$ such that $j = j_p(\lambda)$ for some $p \in [m(\lambda)]$ such that $2i_p(\lambda) = j+1$). Theorem \ref{theoremeHan} can now be reformulated as follows.
\begin{theo} \label{generalizedtheorem}
For all $k \geq 2$, the Dumont-Foata polynomial $F_k(x,y,z)$ has the following combinatorial interpretation:
$$F_{k}(x,y,z) = \sum_{\lambda \in IS_{k+1}} x^{ful(\lambda)}y^{fr(\lambda)}z^{fro(\lambda) + sch(\lambda)}.$$
\end{theo}
\begin{proof}
First of all, maximal points of $f$ are full $k$-sites of $\lambda$: if $f(j) = 2k-2$ then $z_j(f)$ is necessarily defined in the context $(2)$ of Definition \ref{definitionvarphi}, thence $z_j(\lambda) = z_j(f) = f(j)/2 - \lceil j/2 \rceil = k-1-\lceil j/2 \rceil$, and $j \in L(\lambda)$ because otherwise $f(j)$ would equal $2i_p(\lambda) < 2k-2$ for some $p \in [m(\lambda)]$. So $j$ is a full $k$-site of $\lambda$. Reciprocally, if $j \in L(\lambda)$ is such that $z_j(\lambda) = k-1-\lceil j/2 \rceil$, then $f(j) = 2(\lceil j/2 \rceil + z_j(\lambda)) = 2k-2$ so $j$ is a maximal point of $f$. Afterwards, surfixed points of $f$ are odd free $k$-sites and surchained $k$-sites of $\lambda$: if $f(j) = j+1$, then $j = 2i-1$ for some $i \in [k-1]$ and either $j \in L(\lambda)$, in which case $f(j) = 2(i + z_j(\lambda)) = 2i$ thence $j$ is an odd free $k$-site, or $j = j_p(\lambda) = j^p(\lambda)$ for some $p \in [m(\lambda)]$ such that $2i_p(\lambda) = 2i = j+1$, \textit{i.e.}, the integer $j$ is a surchained $k$-site. Reciprocally, if $j$ is an odd free $k$-site then $f(j) = 2(\lceil j/2 \rceil + z_j(\lambda)) = 2(\lceil j/2 \rceil) = j+1$, and if $j$ is a surchained $k$-site then in particular $f(j) = 2i_p(\lambda) = j+1$ for some $p \in [m(\lambda)]$.
As a conclusion, the result comes from Theorem \ref{theoremebijection}.
\end{proof}

\section*{Aknowledgement}

I thank Jiang Zeng for his comments and useful references.
\nocite{*}
\bibliographystyle{alpha}
\bibliography{sample}

\begin{thebibliography}{LLMS13}

\bibitem[Car71]{Carlitz}
Leonard Carlitz.
\newblock {A} conjecture concerning {G}enocchi numbers.
\newblock {\em Norske Vid. Selsk. Skr. (Trondheim)}, 9:4, 1971.

\bibitem[DF76]{DF}
Dominique Dumont and Dominique Foata.
\newblock Une propri\'et\'e de sym\'etrie des nombres de genocchi.
\newblock {\em Bull. Soc. Math. France}, 104:433--451, 1976.

\bibitem[DR94]{Dumont3}
Dominique Dumont and Arthur Randrianarivony.
\newblock {D\'erangements et nombres de Genocchi}.
\newblock {\em Disc. Math.}, 132:37--49, 1994.

\bibitem[Dum72]{Dumont1}
Dominique Dumont.
\newblock {S}ur une conjecture de {G}andhi concernant les nombres de
  {G}enocchi.
\newblock {\em Disc. Math.}, 1:321--327, 1972.

\bibitem[Dum74]{Dumont2}
Dominique Dumont.
\newblock {I}nterpr\'etations combinatoires des nombres de {G}enocchi.
\newblock {\em Duke Math. J.}, 41:305--318, 1974.

\bibitem[Dum95]{Dumont4}
Dominique Dumont.
\newblock {Conjectures sur des sym\'etries ternaires li\'ees aux nombres de
  Genocchi.}
\newblock In FPSAC 1992, editor, {\em Discrete Math.}, number 139, pages
  469--472, 1995.

\bibitem[Han96]{H}
Guo~Niu Han.
\newblock Sym\'etries trivari\'ees sur les nombres de genocchi.
\newblock {\em Europ. J. Combinatorics}, 17:397--407, 1996.

\bibitem[HM11]{HM}
Florent Hivert and Olivier Mallet.
\newblock {C}ombinatorics of $k$-shape and {G}enocchi numbers.
\newblock In FPSAC 2011, editor, {\em Discrete Math. Theor. Comput. Sci.
  Proc.}, pages 493--504, Nancy, 2011.

\bibitem[JV11]{JV}
Matthieu Josuat-Verg\`es.
\newblock {Dumont-Foata polynomials and alternative tableaux}.
\newblock {S\'em. Lothar. Combin. 64, Art. B64b, 17pp}, 2011.

\bibitem[LLMS13]{LLMS}
Thomas Lam, Luc Lapointe, Jennifer Morse, and Marc Shimozono.
\newblock {T}he poset of $k$-shapes and branching rules for $k$-schur
  functions.
\newblock {\em Mem. Amer. Math. Soc.}, 223(1050), 2013.

\bibitem[Mal11]{HM2}
Olivier Mallet.
\newblock {C}ombinatoire des $k$-formes et nombres de {G}enocchi.
\newblock S\'eminaire de combinatoire et th\'eorie des nombres de {l'ICJ}, may
  2011.

\bibitem[{OEI}]{genocchi}
{OEIS Foundation Inc. (2011)}.
\newblock {The On-Line Encyclopedia of Integer Sequences}.
\newblock \url{http://oeis.org/A110501}.

\bibitem[Ran94]{Ran}
Arthur Randrianarivony.
\newblock {Polyn\^omes de Dumont-Foata g\'en\'eralis\'es}.
\newblock {S\'em. Lothar. Combin. 32, Art. B32d, 12pp}, 1994.

\bibitem[RS73]{Riordan}
John Riordan and Paul~R. Stein.
\newblock {P}roof of a conjecture on {G}enocchi numbers.
\newblock {\em Discrete Math.}, 5:381--388, 1973.

\bibitem[Sta99]{Stanley2}
R.P. Stanley.
\newblock {\em {E}numerative {C}ombinatorics}, volume~2.
\newblock Cambridge University Press,Cambridge, 1999.

\bibitem[Zen96]{Zeng}
Jiang Zeng.
\newblock {Sur quelques propri\'et\'es de sym\'etrie des nombres de Genocchi}.
\newblock {\em Disc. Math.}, 153:319--333, 1996.

\end{thebibliography}
\label{sec:biblio}

\end{document}